\documentclass[onefignum,onetabnum]{siamonline220329}
\pdfoutput=1
\usepackage{caption}
\usepackage{subcaption}
\usepackage{enumitem}
\usepackage{amsmath}
\usepackage{amssymb}
\usepackage{bbm}
\usepackage{verbatim}
\usepackage{graphicx,epstopdf}
\usepackage{array}
\usepackage{xfrac}
\usepackage{faktor}
\newcolumntype{R}{>{$}r<{$}} %
\newcolumntype{V}[1]{>{[\;}*{#1}{R@{\;\;}}R<{\;]}} %
\usepackage{braket,amsfonts}
\usepackage{multirow}
\usepackage{booktabs}
\usepackage{graphbox}
\usepackage{overpic}
\usepackage{enumitem}
\usepackage{algpseudocode}
\usepackage{refcount}
\usepackage{mathrsfs}

\renewcommand {\thetable} {\thesection{}.\arabic{table}}
\renewcommand {\thefigure} {\thesection{}.\arabic{figure}}

\usepackage{pgf,tikz,pgfplots}
\pgfplotsset{compat=1.15}
\usetikzlibrary{shapes,arrows}
\usetikzlibrary{calc,patterns,angles,quotes}

\newsiamthm{claim}{Claim}
\newsiamremark{remark}{Remark}
\newsiamremark{hypothesis}{Hypothesis}

\usepackage{graphicx,epstopdf}

\Crefname{ALC@unique}{Line}{Lines}

\usepackage{amsopn}

\usepackage{xspace}
\usepackage{bold-extra}
\usepackage[most]{tcolorbox}

\colorlet{texcscolor}{blue!50!black}
\colorlet{texemcolor}{red!70!black}
\colorlet{texpreamble}{red!70!black}
\colorlet{codebackground}{black!25!white!25}

\lstdefinestyle{siamlatex}{%
  style=tcblatex,
  texcsstyle=*\color{texcscolor},
  texcsstyle=[2]\color{texemcolor},
  keywordstyle=[2]\color{texemcolor},
  moretexcs={cref,Cref,maketitle,mathcal,text,headers,email,url},
}

\tcbset{%
  colframe=black!75!white!75,
  coltitle=white,
  colback=codebackground, 
  colbacklower=white, 
  fonttitle=\bfseries,
  arc=0pt,outer arc=0pt,
  top=1pt,bottom=1pt,left=1mm,right=1mm,middle=1mm,boxsep=1mm,
  leftrule=0.3mm,rightrule=0.3mm,toprule=0.3mm,bottomrule=0.3mm,
  listing options={style=siamlatex}
}

\newtcblisting[use counter=example]{example}[2][]{%
  title={Example~\thetcbcounter: #2},#1}

\newtcbinputlisting[use counter=example]{\examplefile}[3][]{%
  title={Example~\thetcbcounter: #2},listing file={#3},#1}

\DeclareTotalTCBox{\code}{ v O{} }
{ 
  fontupper=\ttfamily\color{black},
  nobeforeafter,
  tcbox raise base,
  colback=codebackground,colframe=white,
  top=0pt,bottom=0pt,left=0mm,right=0mm,
  leftrule=0pt,rightrule=0pt,toprule=0mm,bottomrule=0mm,
  boxsep=0.5mm,
  #2}{#1}

\patchcmd\newpage{\vfil}{}{}{}
\numberwithin{equation}{section}

\def\vec{{\mathrm{vec}}}
\def\diag{{\mathrm{diag}}}
\def\c{\mathbf{c}}

\def\f{\mathbf{f}}

\def\h{\mathbf{h}}
\def\bv{\mathbf{v}}

\def\M{\mathcal{M}}

\def\0{\mathbf{0}}
\def\1{\mathbf{1}}

\def\I{\mathtt{I}}
\def\B{\mathtt{B}}
\def\beps{{\boldsymbol{\varepsilon}}}
\def\blam{{\boldsymbol{\lambda}}}
\def\bsigma{{\boldsymbol{\sigma}}}

\DeclareMathOperator*{\argmin}{argmin}

\title{Convergence Analysis of Volumetric Stretch Energy Minimization and its Associated Optimal Mass Transport
\thanks{Submitted to the editors \today.
\funding{The work of the authors was partially supported by the National Science and Technology Council, the National Center for Theoretical Sciences, and the ST Yau Center in Taiwan. T.-M. Huang, W.-W. Lin, and M.-H. Yueh was partially supported by NSTC 110-2115-M-003-012-MY3, 110-2115-M-A49-004- and 111-2115-M-003-016-, respectively.}}
}

\headers{Convergence Analysis of VSEM and Associated  OMT}{T.-M. Huang, W.-H. Liao, W.-W. Lin, M.-H. Yueh, and S.-T. Yau}

\author{
Tsung-Ming Huang\thanks{Department of Mathematics, National Taiwan Normal University, Taipei, 116, Taiwan (\email{min@ntnu.edu.tw}).}
\and
Wei-Hung Liao\thanks{Department of Applied Mathematics, National Yang Ming Chiao Tung University, Hsinchu, 300, Taiwan (\email{roger2300245@gmail.com}).}
\and
Wen-Wei Lin\thanks{Department of Applied Mathematics, National Yang Ming Chiao Tung University, Hsinchu, 300, Taiwan (\email{wwlin@math.nctu.edu.tw}).}
\and
Mei-Heng Yueh\thanks{Department of Mathematics, National Taiwan Normal University, Taipei, 116, Taiwan (\email{yue@ntnu.edu.tw}).}
\and
Shing-Tung Yau\thanks{Yau Mathematical Sciences Center, Tsinghua University, Beijing, 100084, China (\email{styau@tsinghua.edu.cn}).}
}

\begin{document}
\begin{sloppypar}

\maketitle

\begin{abstract}
The volumetric stretch energy has been widely applied to the computation of volume-/mass-preserving parameterizations of simply connected tetrahedral mesh models. However, this approach still lacks theoretical support. In this paper, we provide the theoretical foundation for volumetric stretch energy minimization (VSEM) to compute volume-/mass-preserving parameterizations. In addition, we develop an associated efficient VSEM algorithm with guaranteed asymptotic R-linear convergence. Furthermore, based on the VSEM algorithm, we propose a projected gradient method for the computation of the volume/mass-preserving optimal mass transport map with a guaranteed convergence rate of $\mathcal{O}(1/m)$, and combined with Nesterov-based acceleration,  the guaranteed convergence rate becomes $\mathcal{O}(1/m^2)$. Numerical experiments are presented to justify the theoretical convergence behavior for various examples drawn from known benchmark models. Moreover, these numerical experiments show the effectiveness and accuracy of the proposed algorithm, particularly in the processing of 3D medical MRI brain images.
\end{abstract}

\begin{keyword}
volume-/mass-preserving parameterization, optimal mass transport, R-linear convergence, projected gradient method, Nesterov-based acceleration, $\mathcal{O}(1/m)$ convergence
\end{keyword}

\begin{MSCcodes}
68U05, 65D18, 52C35, 33F05, 65E10
\end{MSCcodes}

\section{Introduction}

Volume-/mass-preserving parameterization of a 3-manifold $\mathcal{M}$ with a single closed genus-zero boundary by a unit ball $\mathbb{B}^3$, as well as its associated optimal mass transport (OMT), has been widely applied to computer graphics \cite{FlHo05}, digital geometry \cite{HoLe07}, medical image segmentation \cite{LiJY21,LiLH22}, image retrieval \cite{li2013,rubner2000}, image representation and registration \cite{HaZT04,kolouri2017,kolouri2016,wang2013} and generative adversarial networks \cite{LeSC19}. In calculus, we learn that a 2-manifold or 3-manifold can be represented by a given 2D or 3D coordinate system, and then the related curvatures, singularities, maxima, minima, local areas or volume can be computed. In contrast, in practical applications, irregular manifold domains are usually obtained by scanning or sampling data from physical objects. For instance, CT, MRI and PET images are the most common in medical diagnosis, real-time scanning images by satellite for space engineering and weather forecasting, photomicrography in physical and biological engineering, target or obstacle detection for autonomous driving systems and missile navigation, etc. One intuitive idea is to directly calculate numerical solutions for complex problems on irregular manifolds. If the irregular manifold domain is too technically challenging to produce solutions for a complex problem, we should consider the inverse problem of the above problem; that is, the irregular manifold should be parameterized by a regular domain. The most common 3D parametric shapes are a cube or a ball $\mathbb{B}^3$. To this end, efficient algorithms, namely, the volumetric stretch energy minimization (VSEM) method \cite{YuLL19} and OMT methods \cite{GuLS16} for the computation of volume-preserving parameterizations, have recently been highly developed and utilized in applications. In practice, in terms of the effectiveness and accuracy, the VSEM algorithm is much improved compared to the other state-of-the-art algorithms (see \cite{YuLL19} for details). However, VSEM still lacks rigorously mathematical and theoretical support.

In this paper, we first introduce the volumetric stretch energy functional on $\mathcal{M}$ and propose the VSEM algorithm for the computation of the spherical volume-/mass-preserving parameterization between $\mathcal{M}$ and $\mathbb{B}^3$. 
We show that a minimal solution for the volumetric stretch energy functional must be a volume-/mass-preserving map, and vice versa, which successfully supports the setting for our modified volume stretch energy functional. Then, we prove that the VSEM algorithm converges R-linearly under some mild conditions. Next, we consider an early but important OMT problem proposed by Monge in 1781 (see, e.g., \cite{bonnotte2013}) in which a pile of soil is moved from one place to another while preserving the local volume and minimizing the transport cost. Although the set of volume-/mass-preserving maps between $\mathcal{M}$ and $\mathbb{B}^3$ may not be convex, for the discrete OMT problem we still prefer to adopt the projected gradient method to minimize the transport cost for preserving the minimal deformation between $\mathcal{M}$ and $\mathbb{B}^3$. Then, we accelerate the convergence by the Nesterov method \cite{nesterov1983}.
Under mild assumptions of nonexpensiveness and projection properties, the convergence of the projected gradient method can be proven to be a rate of $\mathcal{O}(1/m)$, and with the acceleration, it has a convergence rate of $\mathcal{O}(1/m^2)$.
This numerical algorithm is called the volume-/mass-preserving OMT (VOMT) algorithm. It is effective, reliable and robust.

The main contributions of this paper are threefold.
\begin{enumerate}
\item We prove a fundamental theorem that $f^*$  is the minimal solution of the discrete volumetric stretch energy functional if and only if $f^*$ is volume-/mass-preserving between $\mathcal{M}$ and $\mathbb{B}^3$. Based on this mathematical foundation, we develop a VSEM for the computation of the minimal solution $f^*$ and show the R-linear convergence of VSEM.
\item For the discrete OMT problem, we use the gradient method combined with VSEM as a projector to develop an efficient numerical algorithm, VOMT, for solving the spherical VOMT problem from $\mathcal{M}$ to $\mathbb{B}^3$ with a convergence rate of $\mathcal{O}(1/m)$ and, accelerated by the Nesterov method, with a convergence rate of $\mathcal{O}(1/m^2)$.
\item In practical applications on various benchmarks, numerical experiments confirm that the assumption for R-linear convergence is satisfied, and the related means and standard deviation of the ratios of local volume distortions show the effectiveness and exactness of the proposed VOMT algorithm.
\end{enumerate}

The remaining part of this paper is organized as follows:
In Section \ref{sec:2}, we introduce the discrete $3$-manifold and the volumetric stretch energy. In Section~\ref{sec:3}, we propose a rigorous derivation for the equivalence relationship between the volume-/mass-preserving map and the minimizer of the volumetric stretch energy. In Section \ref{sec:4}, we prove the convergence of the VSEM algorithm in \cite{YuLL19} for the computation of volume-/mass-preserving parameterizations, which provides theoretical support for the VSEM. Then, we introduce the associated VOMT algorithm as well as its convergence analysis in Section \ref{sec:5}. Numerical experiments for the VSEM and VOMT are demonstrated in Section \ref{sec:6} to validate the consistency between the theoretical and numerical results. Concluding remarks are given in Section \ref{sec:7}.

In this paper, we use the following notations:
\begin{itemize}
\item Bold letters, e.g., $\mathbf{f}$, denote real-valued vectors or matrices.
\item Capital letters, e.g., $L$, denote real-valued matrices.
\item Typewriter letters, e.g., $\mathtt{I}$ and $\mathtt{B}$, denote ordered sets of indices.
\item $\mathbf{f}_i$ denotes the $i$th row of the matrix $\mathbf{f}$.
\item $\mathbf{f}^s$ denotes the $s$th column of the matrix $\mathbf{f}$.
\item $\mathbf{f}_\mathtt{I}$ denotes the submatrix of $\mathbf{f}$ composed of $\mathbf{f}_i$, for $i\in\mathtt{I}$.
\item $L_{i,j}$ denotes the $(i,j)$th entry of the matrix $L$.
\item $L_{\mathtt{I},\mathtt{J}}$ denotes the submatrix of $L$ composed of $L_{i,j}$ for $i\in\mathtt{I}$ and $j\in\mathtt{J}$.
\item $\mathbb{R}$ denotes the set of real numbers.
\item $\mathbb{B}^3:=\{ \mathbf{x}\in\mathbb{R}^{3} \mid \|\mathbf{x}\| \leq 1 \}$ denotes the solid ball in $\mathbb{R}^{3}$.
\item $\left[\bv_0, \ldots, \bv_k\right]$ denotes the $k$-simplex with vertices $\bv_0, \ldots, \bv_k$.
\item $|[\bv_0, \ldots, \bv_k]|$ denotes the volume of the $k$-simplex $\left[\bv_0, \ldots, \bv_k\right]$.
\item $\mathbf{0}$ and $\mathbf{1}$ denote the zero and one vectors or matrices of appropriate sizes, respectively.
\item Given two vectors $\mathbf{x}$ and $\mathbf{y}$, $\mathbf{x} \cdot \mathbf{y}$ denotes the inner product $\mathbf{x}^{\top} \mathbf{y}$.
\item Hashtag $\#$ of a set, e.g., $\#(\mathbb{T}(\mathcal{M}))$, denotes the number of elements in $\mathbb{T}(\mathcal{M})$.
\end{itemize}

\section{Discrete 3-manifold and volumetric stretch energy}
\label{sec:2}

Let $\mathcal{M}\subset\mathbb{R}^3$ be a simply connected 3-manifold with a single genus-zero boundary. A discrete 3-manifold model for $\mathcal{M}$ is a simplicial $3$-complex with $n$ vertices
\begin{align*}
\mathbb{V}(\mathcal{M})=\{\bv_t=({v}_t^1, {v}_t^2, {v}_t^3)\in\mathbb{R}^3\}_{t=1}^n,
\end{align*}
and tetrahedra
\begin{align*}
\mathbb{T}(\mathcal{M})=\{[\bv_i, \bv_j, \bv_k, \bv_\ell]\subset\mathbb{R}^3 \text{ for some $\bv_r\in\mathbb{V}(\mathcal{M})$, $r=i, j, k, \ell$}\},
\end{align*}
where the bracket $[\bv_i, \bv_j, \bv_k, \bv_\ell]$ is the 3-simplex (convex hull) of the affinely independent points $\{\bv_r \mid r=i, j, k, \ell\}$. Additionally, the triangular faces and edges of $\mathcal{M}$ are denoted by
\begin{align*}
\mathbb{F}(\mathcal{M})=\{[\bv_i, \bv_j, \bv_k] \mid [\bv_i, \bv_j, \bv_k, \bv_\ell]\in\mathbb{T}(\mathcal{M}) \text{ with some $\bv_\ell\in\mathbb{V}(\mathcal{M})$}\}
\end{align*}
and 
\begin{align*}
\mathbb{E}(\mathcal{M})=\{[\bv_i, \bv_j] \mid [\bv_i, \bv_j, \bv_k]\in\mathbb{F}(\mathcal{M}) \text{ with some $\bv_k\in\mathbb{V}(\mathcal{M})$}\}.
\end{align*}

Because an affine map in $\mathbb{R}^3$ is determined by four independent point correspondences, a piecewise affine map $f:\mathcal{M}\to\mathbb{R}^3$ on a tetrahedral mesh $\mathcal{M}$ can be expressed as an $n\times 3$ matrix defined by the images $f(\bv_t)$ of vertices $\bv_t\in\mathbb{V}(\mathcal{M})$ as
\begin{align} \label{eq:f}
\mathbf{f}:=[\f_1^{\top}, \cdots, \f_n^{\top}]^{\top}\in\mathbb{R}^{n\times 3},    
\end{align}
where $\f_t :=f(\bv_t)=({f}_t^1, {f}_t^2, {f}_t^3)\in\mathbb{R}^{3}$, for $t=1, \dots, n$. We also denote
\begin{align*}
\mathbf{f}=\begin{bmatrix}
\mathbf{f}^1 & \mathbf{f}^2 & \mathbf{f}^3
\end{bmatrix}, \quad \mathbf{f}^s=\begin{bmatrix}
{f}_1^s & \cdots & {f}_n^s
\end{bmatrix}^{\top},\  s=1, 2, 3.
\end{align*}
For a point $\bv\in\mathcal{M}$, $\bv$ must belong to a tetrahedron $\tau$; without loss of generality, let $\tau=[\bv_1, \bv_2, \bv_3, \bv_4]$. Then, the piecewise affine map $f(\bv):\mathbb{T}(\mathcal{M})\to\mathbb{R}^3$ can be expressed as a linear combination of $\{\f_i\}_{i=1}^4$ with barycentric coordinates, i.e.,
\begin{align*}
f\rvert_{\tau}(\bv)=\sum_{i=1}^4\lambda_i f(\bv_i), \quad\lambda_i=\frac{1}{|\tau|}|[\bv_1,\cdots,\widehat{\bv}_i,\cdots, \bv_4]|
\end{align*}
where $\widehat{\bv}_i$ is replaced by $\bv$ for $i=1, \ldots, 4$, and $|\tau|$ denotes the volume of the $3$-simplex $\tau$. The piecewise affine map $f:\mathcal{M}\to\mathbb{B}^3$ is said to be induced by $\mathbf{f}$ and is volume-/mass-preserving if the Jacobian $J_{f^{-1}}=[\frac{\partial f^{-1}}{\partial u^1}, \frac{\partial f^{-1}}{\partial u^2}, \frac{\partial f^{-1}}{\partial u^3}]$ satisfies
\begin{equation} \label{eq:detJ}
\det(J_{f^{-1}}\rvert_{f(\tau)})=1
\end{equation}
with $f(\tau)=[\mathbf{f}_i, \mathbf{f}_j, \mathbf{f}_k, \mathbf{f}_{\ell}]$ for every $\tau =[\bv_i, \bv_j, \bv_k, \bv_{\ell}] \in\mathbb{T}(\mathcal{M})$. Like the derivation in \cite[Appendix A]{YuLL19}, $f:\mathcal{M}\to\mathbb{R}^3$ is volume-/mass-preserving with respect to the tetrahedral volume measure $\mu:\mathbb{T}(\mathcal{M})\to\mathbb{R}_+$ if and only if \eqref{eq:detJ} holds for every $\tau\in\mathbb{T}(\mathcal{M})$. We denote the stretch factor with respect to $\mu$ as
\begin{align}
\sigma_{\mu, f^{-1}}(\tau) = \mu(\tau)/|f(\tau)|.  \label{eq:stretch_factor}
\end{align}
The original VSEM \cite{YuLL19} computes a spherical volume-preserving parameterization between $\mathcal{M}$ and $\mathbb{B}^3$ with $\mu(\tau)=|\tau|$ by minimizing the volumetric stretch energy functional,
\begin{align}
    E_V(f)=\frac{1}{2}\mathrm{trace}(\mathbf{f}^{\top}L_V(f)\mathbf{f}) \label{eq:energy_fun}
\end{align}
where $L_V(f)$ is a volumetric stretch Laplacian matrix with
\begin{subequations} \label{eq:Lap_mtx}
\begin{align}
\label{eq:2}
[L_V(f)]_{ij}=[L_V(f)]_{ij}^{\top}=
\begin{cases}
w_{ij}(f), & \mbox{if $[\bv_i, \bv_j]\in\mathbb{E}(\mathcal{M})$},\\
-\sum_{\ell\neq i}w_{i\ell}(f), & \mbox{if $j=i$},\\
0, & \mbox{otherwise},
\end{cases}
\end{align}
in which, like \cite{YuHL21}, the modified weight $w_{ij}(f)$ is defined by
\begin{align}
w_{ij}(f)
=&-\frac{1}{9}\sum_{\substack{\tau\in\mathbb{T}(\mathcal{M})\\ [\bv_i,\bv_j]\cup[\bv_k, \bv_\ell]\subset\tau\\ [\bv_i,\bv_j]\cap[\bv_k, \bv_\ell]=\emptyset}}\frac{|f([\bv_i, \bv_k, \bv_\ell])||f([\bv_j, \bv_\ell, \bv_k])|\cos\theta_{i, j}^{k, \ell}(f)}{\sigma_{\mu,f^{-1}}(\tau)|f(\tau)|}\nonumber\\
=&-\frac{1}{36}\sum_{\substack{\tau\in\mathbb{T}(\mathcal{M})\\ [\bv_i,\bv_j]\cup[\bv_k, \bv_\ell]\subset\tau\\ [\bv_i,\bv_j]\cap[\bv_k, \bv_\ell]=\emptyset}}\frac{\left(2|f([\bv_i, \bv_k, \bv_\ell])|\right)\left(2|f([\bv_j, \bv_k, \bv_\ell])|\right)\cos\theta_{i, j}^{k, \ell}(f)}{\mu(\tau)}\nonumber\\
=&-\frac{1}{36}\sum_{\substack{\tau\in\mathbb{T}(\mathcal{M})\\ [\bv_i,\bv_j]\cup[\bv_k, \bv_\ell]\subset\tau\\ [\bv_i,\bv_j]\cap[\bv_k, \bv_\ell]=\emptyset}}\frac{[(\mathbf{f}_k-\mathbf{f}_i)\times(\mathbf{f}_{\ell}-\mathbf{f}_i)]^\top[(\mathbf{f}_{\ell}-\mathbf{f}_j)\times(\mathbf{f}_k-\mathbf{f}_j)]}{\mu(\tau)}.\label{eq:3-3}
\end{align}
\end{subequations}
Here, $\theta_{i, j}^{k, \ell}(f)$ is the dihedral angle between the triangular faces $f([\bv_i, \bv_k, \bv_\ell])$ and $f([\bv_j, \bv_\ell, \bv_k])$ in tetrahedron $f(\tau)$. In Sections~\ref{sec:3} and \ref{sec:4} we will provide the theoretical foundation of SVEM.

\begin{remark}
In practice, the dihedral angle $\theta_{i,j}^{k,\ell}(f)$ is computed by using the identity
$$
\cos(\pi-\theta_{i,j}^{k,\ell}(f)) = \mathbf{n}_{i,k,\ell}(f)^\top \mathbf{n}_{j,k,\ell}(f),
$$
where $\mathbf{n}_{i,j,k}(f)$ denotes the unit normal vector of face $f([v_i,v_j,v_k])$.
\end{remark}

\section{Volume-/mass-preserving parameterization vs. volumetric stretch energy minimizer}
\label{sec:3}

Let $\mathcal{M}$ be a simply connected 3-manifold with a genus-zero boundary. We consider the volumetric stretch energy functional on $\mathcal{M}$ as in \eqref{eq:energy_fun},
\begin{align}
\label{eq:1}
E_V(f)=\frac{1}{2}\textrm{trace}\left(\mathbf{f}^{\top}L_V(f)\mathbf{f}\right)=\frac{1}{2}\sum_{s=1}^3 {\mathbf{f}^s}^{\top}L_V(f)\mathbf{f}^s,
\end{align}
where $L_V(f)$ is the volumetric stretch Laplacian matrix as in \eqref{eq:Lap_mtx}.

In \Cref{subsec:3.1}, we first show the equivalence of volumetric stretch energy minimizers and volume-/mass-preserving parameterizations. 
In \Cref{subsec:3.2}, we provide a neat gradient formula of $E_V$ so that the minimizers of $E_V$ with a fixed spherical boundary constraint can be conveniently derived in \Cref{{subsec:min_VSEM_fb}}. 

To simplify the derivations in Subsections \ref{subsec:3.1} and \ref{subsec:3.2}, we denote the volumetric stretch energy restricted to a tetrahedron $\tau \in \mathbb{T}(\mathcal{M})$ as $E_{\tau}(f(\tau))$ with Laplacian matrix $L_{\tau}(f) \equiv L_V(f)\rvert_{\tau}$. Then, the volumetric stretch energy functional $E_V(f)$ in \eqref{eq:1} is equal to the summation of all $E_{\tau}(f(\tau))$. According to \eqref{eq:Lap_mtx}, we give a new representation of $L_{\tau}(f)$ as follows.

Without loss of generality, let $\tau = [\bv_1,\bv_2,\bv_3,\bv_4]$ be a tetrahedron in $\mathbb{T}(\M)$. 
Then, the image of $f(\tau)$ is  $[\f_1, \f_2, \f_3, \f_4]$.
By substituting \eqref{eq:3-3} into $L_V(f)\rvert_{\tau}$ of \eqref{eq:2}, we obtain
\begin{subequations}\label{eq:8}
\begin{align}
\label{eq:8_1}
L_{\tau}(f)= - \frac{1}{36\mu(\tau)}\left[a_{ij}\right] \in \mathbb{R}^{4 \times 4}, 
\end{align}
where, for $[\bv_i,\bv_j]\cap[\bv_k, \bv_\ell]=\emptyset$ and $i, j, k, \ell \in\{1, 2, 3, 4\}$,
\begin{align}
    a_{ij}=[(\mathbf{f}_{k}-\mathbf{f}_{i})\times(\mathbf{f}_{\ell}-\mathbf{f}_{i})]^\top[(\mathbf{f}_{\ell}-\mathbf{f}_{j})\times(\mathbf{f}_{k}-\mathbf{f}_{j})], \quad a_{ij}=a_{ji}\label{eq:8_2}
\end{align}
and
\begin{align}
a_{ii}=-\sum_{\substack{j\neq i}}a_{ij}.
\end{align}
\end{subequations}
Applying the fundamental identity
\begin{align*}
(\mathbf{a}\times\mathbf{b})^\top(\mathbf{x}\times\mathbf{y})=(\mathbf{a}^\top\mathbf{x})(\mathbf{b}^\top\mathbf{y})-(\mathbf{a}^\top\mathbf{y})(\mathbf{b}^\top\mathbf{x}),
\end{align*}
$a_{ij}$ of \eqref{eq:8_2} can be expanded into the form
\begin{align}\label{fundId}
a_{ij}
=&-[(\f_{k}-\f_{i})^\top(\f_{k}-\f_{j})][(\f_{\ell}-\f_{i})^\top(\f_{\ell}-\f_{j})]\\
& +[(\f_{k}-\f_{i})^\top(\f_{\ell}-\f_{j})][(\f_{\ell}-\f_{i})^\top(\f_{k}-\f_{j})].\nonumber
\end{align}  

\subsection{Equivalence of volumetric stretch energy minimizers and volume-/mass-preserving parameterizations}
\label{subsec:3.1}

First, we provide a geometric interpretation of $E_V$ in the following theorem, which is the crucial step for the proof of minimizers of $E_V$ being volume-/mass-preserving, and vice versa.

\begin{theorem} \label{thm:E_V}
The volumetric stretch energy functional \eqref{eq:1} can be reformulated as
\begin{align}
\label{Vfsem}
E_V(f)=\sum_{\tau\in\mathbb{T}(\mathcal{M})}\frac{3 |f(\tau)|^2}{2\mu(\tau)}.
\end{align}
\end{theorem}
\begin{proof}
Note that the image volume $|f(\tau)|$ can be written as
\begin{align}
|f(\tau)| = \frac{1}{6} \left( ((\f_2-\f_1)\times(\f_3-\f_1))^\top(\f_4-\f_1) \right). \label{eq:vol_f}
\end{align}
Together with the formula of $L_\tau(f)$ in \eqref{eq:8_1}, \eqref{fundId} and \eqref{eq:vol_f}, we can show that the following equation holds 
\begin{align}\label{eq:3.6}
\frac{1}{2} \sum_{s=1}^3 {\f_\tau^s}^\top L_\tau(f) \f_\tau^s =\frac{3 |f(\tau)|^2}{2\mu(\tau)}
\end{align}
by a direct expansion of both sides of \eqref{eq:3.6} with the symbolic toolbox of MATLAB.

Summing over all tetrahedra in $\mathbb{T}(\M)$, we obtain
\begin{align*}
E_V(f) &= \frac{1}{2} \sum_{s=1}^3 \sum_{\tau\in\mathbb{T}(\M)} {\f_\tau^s}^\top L_\tau(f) \f_\tau^s 
=\sum_{\tau\in\mathbb{T}(\mathcal{M})}\frac{3 |f(\tau)|^2}{2\mu(\tau)}.
\end{align*}
\end{proof}

Theorem~\ref{thm:E_V} indicates that the volumetric stretch energy $E_V$ can be represented solely by $\mu(\tau)$ and the image volume $|f(\tau)|$, where $\tau\in\mathbb{T}(\M)$. 
In the following theorem, we further prove that the minimizer of the volumetric stretch energy is volume-/mass- parameterization, and vice versa.

\begin{theorem} \label{thm:f_argmin}
Let $\mathcal{M}\subset\mathbb{R}^3$ be a simply connected 3-manifold with a genus-zero boundary. 
Under the constraint that the image volume is the same as the volume/mass of $\M$, the map $f^*$ is a minimizer of the volumetric stretch energy functional if and only if $f^*$ is volume-/mass-preserving, i.e.,
\begin{align}
f^*=\mathop{\argmin}_{|f(\M)|=\mu(\M)}E_V(f)\quad \Longleftrightarrow\quad \mu(\tau)=|f^*(\tau)|, \quad\text{for every $\tau\in\mathbb{T}(\mathcal{M})$}. \label{eq:equ_cond}
\end{align}
\end{theorem}
\begin{proof}
Let $q:=\#\left(\mathbb{T}(\mathcal{M})\right)$. 
Without loss of generality, we normalize the total volume/mass to one, i.e.,  
\begin{align}
1=\mu(\mathcal{M})=\sum_{r=1}^q\mu(\tau_r),\quad 1=|f(\mathcal{M})|=\sum_{r=1}^q|f(\tau_r)|. \label{eq:massEQ1}
\end{align}
We denote ${w}_r=|f(\tau_r)|\in(0,1)$ and ${u}_r=\mu(\tau_r)\in(0,1)$, for $1\leq r\leq q$.
Then, by \eqref{Vfsem}, the optimal problem 
\begin{align*}
    \min_{|f(\M)|=\mu(\M) = 1}  E_V(f)  
\end{align*}
can be rewritten as
\begin{subequations} \label{eq:opt_Eng_fun}
\begin{align}
    \min \hspace{5mm} & \ E_V({w}_1, {w}_2,\cdots, {w}_q) = \frac{3}{2}\sum_{r=1}^q\frac{{w}_r^2}{{u}_r} \label{eq:opt_Eng_fun_cost} \\
    \mbox{subject to } & \ \sum_{r=1}^q{w}_r=1. \label{eq:opt_Eng_fun_st}
\end{align}
\end{subequations}
The Karush--Kuhn--Tucker (KKT) conditions of \eqref{eq:opt_Eng_fun} imply that
\begin{subequations} \label{eq:KKT}
\begin{align}
    \frac{3w_r}{u_r} + \lambda &= 0, \quad r = 1, \ldots, q, \label{eq:KKT_1}\\
    \sum_{r=1}^q w_r - 1 &= 0,
\end{align}
\end{subequations}
where $\lambda$ is a Lagrange multiplier.
Using the results in \eqref{eq:massEQ1} and \eqref{eq:KKT_1}, we have $\lambda = -3$.
Substituting $\lambda=-3$ into \eqref{eq:KKT_1}, we obtain ${w}_r={u}_r$, $1\leq r\leq q$. Furthermore, the energy function $E_V({w}_1, {w}_2,\cdots, {w}_q)$ in \eqref{eq:opt_Eng_fun_cost} is convex and the associated constraint in \eqref{eq:opt_Eng_fun_st} is a convex set, which implies that the solution of KKT conditions in \eqref{eq:KKT} is the minimizer of the energy functional.
\end{proof}

\begin{remark} \label{rem:OptVal_Ev}
According to Theorems~\ref{thm:E_V} and \ref{thm:f_argmin}, if $f^{*} \in \mathbb{B}^3$ is the minimizer in Theorem~\ref{thm:f_argmin} with $| f^{*}(\mathcal{M}) | = \mu(\mathcal{M}) = \frac{4}{3} \pi$, then
\begin{align*}
    E_V(f^{*}) = \frac{3}{2} \sum_{\tau\in\mathbb{T}(\mathcal{M})}\frac{ |f^{*}(\tau)|^2}{\mu(\tau)} = \frac{3}{2} \sum_{\tau\in\mathbb{T}(\mathcal{M})} |f^{*}(\tau)| = \frac{3}{2} | f^{*}(\mathcal{M}) | = 2 \pi.
\end{align*}
In Table~\ref{tab:VSEM}, we will check the volumetric stretch energy $E_V(f^{*})$, computed by Algorithm~\ref{alg:VSEM}, for various benchmark examples is close to $2\pi$. 
\end{remark}

With the conclusion of \Cref{thm:f_argmin}, volume-/mass-preserving parameterizations can be computed by minimizing $E_V$ in \eqref{eq:equ_cond}. In the following subsection, we provide a neat gradient formula of $E_V$ so that the computation of minimizers in \Cref{thm:f_argmin} can be conveniently carried out.

\subsection{Gradient of the volumetric stretch energy}
\label{subsec:3.2}

For the computation of the gradient 
\begin{align}
\label{eq:5}
\nabla_{\mathbf{f}}E_V(f):=
\begin{bmatrix}
\nabla_{\mathbf{f}^1}E_V(f)\\
\nabla_{\mathbf{f}^2}E_V(f)\\
\nabla_{\mathbf{f}^3}E_V(f)
\end{bmatrix},
\end{align}
of \eqref{eq:1}, from \eqref{eq:2}, we write $E_V(f)$ as
\begin{align*}
E_V(f)=\frac{1}{2}\mathrm{trace}\left(\sum_{\tau \in \mathbb{T}(\mathcal{M})}{\mathbf{f}_{\tau}}^{\top} L_{\tau}(f)\mathbf{f}_{\tau} \right),
\end{align*}
where $L_{\tau}(f)$ is defined in \eqref{eq:8}.
For the $t$-th vertex, $1\leq t\leq n$,  in \eqref{eq:5}, the gradient of $E_V(f)$ on $\mathbf{f}^s$, $s=1, 2, 3$, is computed by
\begin{align}
[\nabla_{\mathbf{f}^s}E_V(f)]_t 
&=[L_V(f)]_{t,1:n} \f^{s} + 
\frac{1}{2}\mathrm{trace}\left(\sum_{\tau \in \mathbb{T}(\mathcal{M})} \mathbf{f}_{\tau}^{\top}\frac{\partial L_{\tau}(f)}{\partial f_t^s}\mathbf{f}_{\tau} \right).\label{eq:6}
\end{align}

To determine the gradient $\nabla_{\mathbf{f}^s}E_V(f)$, we first expand and simplify the second term in \eqref{eq:6}.
For $s\in\{1, 2, 3\}$ and $t\in \{1, 2, 3, 4\}$, the partial derivative of the volumetric stretch Laplacian $L_{\tau}(f)$ in \eqref{eq:8_1} is given by
\begin{align}
\label{eq:10}
\frac{\partial}{\partial f_t^s}L_{\tau}(f)=-\frac{1}{36\mu(\tau)}
\begin{bmatrix}
\frac{\partial a_{11}}{\partial f_t^s}&\frac{\partial a_{12}}{\partial f_t^s}&\frac{\partial a_{13}}{\partial f_t^s}&\frac{\partial a_{14}}{\partial f_t^s}\\
\frac{\partial a_{21}}{\partial f_t^s}&\frac{\partial a_{22}}{\partial f_t^s}&\frac{\partial a_{23}}{\partial f_t^s}&\frac{\partial a_{24}}{\partial f_t^s}\\
\frac{\partial a_{31}}{\partial f_t^s}&\frac{\partial a_{32}}{\partial f_t^s}&\frac{\partial a_{33}}{\partial f_t^s}&\frac{\partial a_{34}}{\partial f_t^s}\\
\frac{\partial a_{41}}{\partial f_t^s}&\frac{\partial a_{42}}{\partial f_t^s}&\frac{\partial a_{43}}{\partial f_t^s}&\frac{\partial a_{44}}{\partial f_t^s}
\end{bmatrix}.
\end{align}
From \eqref{fundId},
for $[\bv_i, \bv_j]\cap[\bv_k, \bv_{\ell}]=\emptyset$ and $i, j, k, \ell\in \{1, 2, 3, 4\}$, we have
\begin{subequations}\label{eq:11}
\begin{align}
\frac{\partial a_{ij}}{\partial \f_{i}}&=(\f_{k}-\f_{j})[(\f_{k}-\f_{\ell})^\top(\f_{j}-\f_{\ell})]+(\f_{\ell}-\f_{j})[(\f_{\ell}-\f_{k})^\top(\f_{j}-\f_{k})], \\
\frac{\partial a_{ij}}{\partial \f_{k}}
&=(\f_{i}+\f_{j}-2\f_{k})[(\f_{i}-\f_{\ell})^\top(\f_{j}-\f_{\ell})]\\
&-(\f_{j}-\f_{\ell})[(\f_{i}-\f_{\ell})^\top(\f_{j}-\f_{k})]-(\f_{i}-\f_{\ell})[(\f_{i}-\f_{k})^\top(\f_{j}-\f_{\ell})]\nonumber.
\end{align}
\end{subequations} 
By using $a_{ij} = a_{ji}$, \eqref{eq:vol_f}, \eqref{eq:10} and \eqref{eq:11}, a direct computation  with the symbolic toolbox of MATLAB yields 
\begin{align}
{\mathbf{f}^r_{\tau}}^{\top}\frac{\partial}{\partial f_t^s}L_{\tau}(f)\mathbf{f}^r_{\tau}=
\begin{cases}
0, & \text{if $r=s$},\\
l(s, t), & \text{if $r\neq s$},
\end{cases} \label{eq:fTpartf_b}
\end{align}
where 
\begin{align*}
l(s, i) &=
12\left(f_j^{s-1}(f_k^{s+1}-f_\ell^{s+1})+f_k^{s-1}f_\ell^{s+1}-f_k^{s+1}f_\ell^{s-1}-f_j^{s+1}(f_k^{s-1}-f_\ell^{s-1})\right) |f(\tau)|,  \\ 
l(s, j) &=
-12\left(f_i^{s-1}(f_k^{s+1}-f_\ell^{s+1})+f_k^{s-1}f_\ell^{s+1}-f_k^{s+1}f_\ell^{s-1}-f_i^{s+1}(f_k^{s-1}-f_\ell^{s-1})\right) |f(\tau)|. 
\end{align*}

The following theorem gives a simple formula for the gradient of $E_V$, which provides a foundation for the VSEM to compute critical points of $E_V$.

\begin{theorem} \label{thm:GradEv}
The gradient of $E_V$ in \eqref{eq:1} can be formulated as
\begin{align*}
\nabla_\f E_V(f) \equiv
\begin{bmatrix}
\nabla_{\f^1} E_V(f) \\
\nabla_{\f^2} E_V(f) \\
\nabla_{\f^3} E_V(f) 
\end{bmatrix}
= 3\begin{bmatrix}
L_V(f) \f^1 \\
L_V(f) \f^2 \\
L_V(f) \f^3 
\end{bmatrix}.
\end{align*}
\end{theorem}
\begin{proof} 
It can be verified by the symbolic toolbox of MATLAB that 
$$
[L_\tau(f) \f_\tau^s]_t = \frac{1}{2} l(s,t), \quad\text{for $t\in\{1,2,3,4\}$},
$$
where $L_{\tau}(f)$ is given in \eqref{eq:8_1}.
As a result, \eqref{eq:fTpartf_b} indicates that
$$
\frac{1}{2} \sum_{r=1}^3 {\f_\tau^r}^\top \frac{\partial}{\partial f_t^s} [L_\tau(f)] \f_\tau^r 
= l(s,t)
= 2 [L_\tau(f) \f_\tau^s]_t.
$$
By the summation over all the tetrahedra $\tau\in\mathbb{T}(\mathcal{M})$, we have
\begin{equation} \label{eq:GradEv0}
\frac{1}{2}\mathrm{trace}\left(\sum_{\tau \in \mathbb{T}(\mathcal{M})} \mathbf{f}_{\tau}^{\top}\frac{\partial L_{\tau}(f)}{\partial f_t^s}\mathbf{f}_{\tau} \right) = 2 [L_V(f) \f^s]_t.
\end{equation}
Thus, by substituting \eqref{eq:GradEv0} into \eqref{eq:6}, we obtain
\begin{align*}
[\nabla_{\f^s} E_V(f)]_t = 3 [L_V(f) \f^s]_t.
\end{align*}
\end{proof}

\subsection{Minimizer of the volumetric stretch energy with fixed boundary points} \label{subsec:min_VSEM_fb}
From Theorem~\ref{thm:GradEv}, the Euler--Lagrange equations corresponding to $E_V(f)$ in \eqref{eq:1} are
\begin{align*}
L_{V}(f)\mathbf{f}^s = \0,\quad s=1, 2, 3,
\end{align*}
which has only a trivial solution $[\mathbf{f}^1, \mathbf{f}^2, \mathbf{f}^3] = [\1, \1, \1]$ because of $L_V(f)$ being a singular $M$-matrix. Such trivial solution does not satisfy the condition $| f(\mathcal{M}) | = \mu(\mathcal{M}) $ in Theorem~\ref{thm:f_argmin}. To get the minimizer $f^*$ in Theorem~\ref{thm:f_argmin}, we consider the optimal problem for the minimal volumetric stretch energy with fixed spherical boundary points so that $| f(\mathcal{M}) | = \frac{4}{3} \pi$ as stated in Remark~\ref{rem:OptVal_Ev}.

Let
\begin{align*}
\mathtt{B} = \{t \,|\, \bv_t\in\partial\mathcal{M}\}\quad\text{and}\quad \mathtt{I} = \{1, \ldots, n\} \backslash\mathtt{B}.
\end{align*}
Under a given spherical boundary constraint $\f^{*s}_{\B}$, we consider the optimal problem:
\begin{align}\label{eq:NLP_VSEM}
    \min\{E_V(f)\rvert\  \f^s_\B = \f^{*s}_{\B}, s=1, 2, 3\}.
\end{align}
The associated Lagrange function of \eqref{eq:NLP_VSEM} is
\begin{align*}
    g(\f_{\I}, \f_{\B}, \blam_{\B}) = E_V(f) + \sum_{s=1}^3\blam_{\B}^{s\top} (\f^s_{\B} - \f^{*s}_{\B}),
\end{align*}
where $\blam_{\B}=[\blam_{\B}^1, \blam_{\B}^2, \blam_{\B}^3]$.
By the KKT conditions of \eqref{eq:NLP_VSEM} and Theorem~\ref{thm:GradEv}, we have
\begin{subequations}
\begin{align}
    0 &= \nabla_{\f^s_{\I}} g(\f_{\I}, \f_{\B}, \blam_{\B}) = L_{\I,\I}(f) \f^s_{\I} + L_{\I,\B}(f) \f^s_{\B},\label{eq:3.18a}\\
    0&= \nabla_{\f^s_{\B}} g(\f_{\I}, \f_{\B}, \blam_{\B}) = 3L_{\B,\I}(f) \f^s_{\I} + 3L_{\B,\B}(f) \f^s_{\B} + \blam^s_{\B},\\
    0&= \nabla_{\blam^s_{\B}} g(\f_{\I}, \f_{\B}, \blam_{\B}) = \f^s_{\B} - \f^{*s}_{\B}.
\end{align}
\end{subequations}
If $\f^{*s}_{\I}$ solves the equation \eqref{eq:3.18a} with $\f^s_{\B}=\f^{*s}_{\B}$, for $s=1, 2, 3$, then by Theorem~\ref{thm:f_argmin}, $f^{*}$ induced by $\f^{*} \equiv [\f^{*\top}_{\I}, \f^{*\top}_{\B}]^{\top}$ is a volume-/mass-preserving map.

\section{Convergence of the volumetric stretch energy minimization}
\label{sec:4}

\begin{algorithm}
\caption{VSEM for the volume-preserving parameterization} \label{alg:VSEM}
\begin{algorithmic}[1]
\Require A simply connected tetrahedral mesh $\mathcal{M}$ and a tolerance $\varepsilon$.
\Ensure A volume-preserving parameterization $f:\mathcal{M}\to\mathbb{B}^3$ induced by $\mathbf{f}$.
\State Let $n$ be the number of vertices of $\mathcal{M}$.
\State Let $\mathtt{B} = \{t \,|\, v_t\in\partial\mathcal{M}\}$ and $\mathtt{I} = \{1, \ldots, n\} \backslash\mathtt{B}$.
\State Compute a spherical area-preserving parameterization $\mathbf{f}_\mathtt{B}$ by the SEM algorithm in \cite{YuLL19}.
\State Compute $\mathbf{f} = \begin{bmatrix}\mathbf{f}^1 & \mathbf{f}^2 & \mathbf{f}^3\end{bmatrix}$ by solving the linear systems
\[
[L_{{V}}]_{\mathtt{I},\mathtt{I}} \mathbf{f}_\mathtt{I}^{s} = -[L_{{V}}]_{\mathtt{I},\mathtt{B}} \mathbf{f}_\mathtt{B}^{s}, 
\quad s = 1,2,3,
\]
where $L_{{V}}$ is defined by $L_{{V}}(\mathrm{id})$ and $\mathbf{f}^s = \begin{bmatrix}
\mathbf{f}^s_{\mathtt{I}}\\
\mathbf{f}^s_{\mathtt{B}}
\end{bmatrix}$, $s=1,2,3$.
\State Let $\delta\gets \infty$ and $\widehat{\mathbf{f}}\gets\mathbf{f}$ ($\widehat{f}\gets f$, where $\widehat{f}$ is induced by $\widehat{\mathbf{f}}$).
\While{$\delta>\varepsilon$}
\State Update $L_{{V}}(\widehat{f})$, where $L_{{V}}(\widehat{f})$ is defined as \eqref{eq:2} with $w_{i,j}(f)$ in \eqref{eq:3-3}.
\State Update $\mathbf{f}=\begin{bmatrix}\mathbf{f}^1 & \mathbf{f}^2 & \mathbf{f}^3\end{bmatrix}$ by solving the linear systems \label{alg:Lap_LS}
\begin{align*}
[L_{{V}}]_{\mathtt{I},\mathtt{I}} \mathbf{f}_\mathtt{I}^{s} = -[L_{{V}}]_{\mathtt{I},\mathtt{B}} \mathbf{f}_\mathtt{B}^{s}, 
\quad s=1,2,3.
\end{align*}
\State Update $\delta \gets E_{{V}}(\widehat{f})-E_{{V}}(f)$ and $\widehat{\mathbf{f}}\gets \mathbf{f}$ ($\widehat{f}\gets f$).
\EndWhile
\State \Return $\mathbf{f}$.
\end{algorithmic}
\end{algorithm}

The VSEM for the volume-/mass-preserving parameterization with a fixed spherical area-/mass-preserving boundary constraint as stated in Algorithm~\ref{alg:VSEM} was proposed in \cite{YuLL19}. It is a novel and efficient algorithm. However, there is a lack of rigorous theoretical support. In this section, we fill the gaps in theory.

Given Theorem~\ref{thm:f_argmin}, the minimizer $f^{\ast} \in \mathbb{B}^3$ with the constraint $|f^*(\M)|=\mu(\M)$ of the volumetric stretch energy functional in \eqref{eq:energy_fun} is the volume-/mass-preserving map from $\mathcal{M}$ to $\mathbb{B}^3$. 
Under a given spherical boundary constraint $\f^*_\B$, the equation \eqref{eq:3.18a} indicates that $f^{\ast}$ would satisfy
\begin{equation} \label{eq:Lap_LS}
[L_V(f^*)]_{\I,\I} \f^*_\I = -[L_V(f^*)]_{\I,\B} \f^*_\B.
\end{equation}
This result tell us that the map $f^*:\mathcal{M} \to \mathbb{B}^3$, produced by Algorithm~\ref{alg:VSEM}, is volume-/mass-preserving if the algorithm is convergent.

Next, we give a mathematical analysis of the convergence for Algorithm~\ref{alg:VSEM}. 
At the $(m+1)$th step of Algorithm~\ref{alg:VSEM}, we denote
\begin{align*}
L_{\I}^{(m)} = [L_V(f^{(m)})]_{\I, \I}, \quad L_{\B}^{(m)} = [L_V(f^{(m)})]_{\I, \B}.
\end{align*}
From \eqref{eq:Lap_LS} or Step~\ref{alg:Lap_LS} of Algorithm~\ref{alg:VSEM}, we have
\begin{align}
      \f_{\I}^{s(m+1)} = - ( L_{\I}^{(m)})^{-1} L_{\B}^{(m)} \f_{\B}^{*s}, \label{eq:f_Is}
\end{align}
which implies that
\begin{subequations} \label{eq:eps_form}
\begin{align}
\beps^{s(m+1)} 
&\equiv  \f_{\I}^{s(m+1)} - \f_{\I}^{s(m)} = \left( ( L_{\I}^{(m-1)})^{-1} L_{\B}^{(m-1)} - ( L_{\I}^{(m)})^{-1} L_{\B}^{(m)} \right) \f_{\B}^{*s} \nonumber \\
&= \left( ( L_{\I}^{(m-1)} )^{-1} -( L_{\I}^{(m)} )^{-1}  \right) L_{\B}^{(m-1)} \f_{\B}^{*s} - ( L_{\I}^{(m)} )^{-1} \left( L_{\B}^{(m)} - L_{\B}^{(m-1)} \right) \f_{\B}^{*s} \nonumber \\
&= ( L_{\I}^{(m)} )^{-1} \left(  L_{\I}^{(m)}  - L_{\I}^{(m-1)}  \right) ( L_{\I}^{(m-1)} )^{-1} L_{\B}^{(m-1)} \f_{\B}^{*s} - ( L_{\I}^{(m)} )^{-1} \left( L_{\B}^{(m)} - L_{\B}^{(m-1)} \right) \f_{\B}^{*s} \nonumber \\
&= ( L_{\I}^{(m)} )^{-1} \begin{bmatrix} 
L_{\I}^{(m)}  - L_{\I}^{(m-1)}  & L_{\B}^{(m)} - L_{\B}^{(m-1)}
\end{bmatrix} \mathbf{g}^{(m)}_s \nonumber \\
&= \left( L_{\I}^{(m)} \right)^{-1} \left( \begin{bmatrix} 
L_{\I}^{(m)} & L_{\B}^{(m)} 
\end{bmatrix} - \begin{bmatrix} 
L_{\I}^{(m-1)} & L_{\B}^{(m-1)} 
\end{bmatrix} \right) \mathbf{g}^{(m)}_s \label{eq:error_eps_m}
\end{align}
for $s = 1, 2, 3$, where
\begin{align}
\mathbf{g}^{(m)}_s = \begin{bmatrix}
( L_{\I}^{(m-1)} )^{-1}  L_{\B}^{(m-1)} \f_{\B}^{*s} \\ - \f_{\B}^{*s}  
\end{bmatrix}. \label{eq:g_m}  
\end{align}
\end{subequations}
Equation \eqref{eq:error_eps_m} indicates that the components of $\beps^{s(m+1)}$ depend on the elements $w_{ij}(f^{(m)})-w_{ij}(f^{(m-1)})$ for $i, j = 1, \ldots, n$. Based on the result in \eqref{eq:3-3}, we derive a new representation of $w_{ij}(f^{(m)})-w_{ij}(f^{(m-1)})$ in the following lemma.
\begin{lemma} \label{lem:w_ij}
Let $\f^{(m-1)}\equiv \f^{(m-1)}_{\I}$ and $\f^{(m)}\equiv \f^{(m)}_{\I}$ be computed by \eqref{eq:f_Is}.
$w_{ij}(f^{(m-1)})$ and $w_{ij}(f^{(m)})$ are the $(i, j)$th element of the Laplacian matrices $L_V(f^{(m-1)})$ and $L_V(f^{(m)})$, respectively, in \eqref{eq:2}. Define
\begin{align}
\beps^{(m)}_t= (\f^{(m)}_t- \f^{(m-1)}_t)^{\top}, \quad t=1,\ldots, n. \label{eq:def_eps_t}
\end{align}
For $[\bv_i, \bv_j]\in\mathbb{E}(\mathcal{M})$, $\tau\equiv [\bv_i, \bv_j, \bv_k, \bv_\ell]\in\mathbb{T}(\mathcal{M})$, the vectors $\h_{pq}^{(\cdot)}$ are defined by $\h_{pq}^{(\cdot)}\equiv \f_p^{(\cdot)}-\f_q^{(\cdot)}$ for $p, q\in\{i, j, k, \ell\}$. 
Then,
\begin{align*}
\widehat{w}_{ij}^{(m)} &\equiv w_{ij}(f^{(m)})-w_{ij}(f^{(m-1)}) \nonumber \\
&=\sum_{\substack{\tau\in\mathbb{T}(\mathcal{M})\\ [\bv_i,\bv_j]\cup[\bv_k, \bv_{\ell}]\subset\tau\\ [\bv_i,\bv_j]\cap[\bv_k, \bv_{\ell}]=\emptyset}}\left(-\c_i^{(m)}\cdot \beps_i^{(m)}-\c_j^{(m)} \cdot \beps_j^{(m)}+\c_k^{(m)} \cdot \beps_k^{(m)}+\c_{\ell}^{(m)} \cdot \beps_{\ell}^{(m)}\right),
\end{align*}
where
\begin{subequations} \label{eq:c_ijkl}
\begin{align}
\c_i^{(m)}=-\frac{1}{36 \mu(\tau)}
&\left((\h_{{\ell}i}^{(m)}\cdot\h_{kj}^{(m)})\h_{{\ell}j}^{(m)}
+(\h_{ki}^{(m-1)}\cdot\h_{{\ell}j}^{(m-1)})\h_{kj}^{(m)}\right. \nonumber \\
&\quad\left.+(\h_{{\ell}i}^{(m)}\cdot\h_{{\ell}j}^{(m)})\h_{kj}^{(m)}
+(\h_{ki}^{(m-1)}\cdot\h_{kj}^{(m-1)})\h_{{\ell}j}^{(m)}\right),\\
\c_j^{(m)}=-\frac{1}{36 \mu(\tau)}
&\left((\h_{{\ell}i}^{(m)}\cdot\h_{kj}^{(m)})\h_{ki}^{(m-1)}
+(\h_{ki}^{(m-1)}\cdot\h_{{\ell}j}^{(m-1)})\h_{{\ell}i}^{(m-1)}\right. \nonumber \\
&\quad\left.+(\h_{{\ell}i}^{(m)}\cdot\h_{{\ell}j}^{(m)})\h_{ki}^{(m-1)}
+(\h_{ki}^{(m-1)}\cdot\h_{kj}^{(m-1)})\h_{{\ell}i}^{(m-1)}\right), \\
\c_k^{(m)}=-\frac{1}{36 \mu(\tau)}
&\left((\h_{{\ell}i}^{(m)}\cdot\h_{kj}^{(m)})\h_{{\ell}j}^{(m)}
+(\h_{ki}^{(m-1)}\cdot\h_{{\ell}j}^{(m-1)})\h_{{\ell}i}^{(m-1)}\right. \nonumber \\
&\quad\quad\left.+(\h_{{\ell}i}^{(m)}\cdot\h_{{\ell}j}^{(m)})(\h_{kj}^{(m)}+\h_{ki}^{(m-1)}\right), \\
\c_{\ell}^{(m)}=-\frac{1}{36 \mu(\tau)}
&\left((\h_{{\ell}i}^{(m)}\cdot\h_{kj}^{(m)})\h_{ki}^{(m-1)}
+(\h_{ki}^{(m-1)}\cdot\h_{{\ell}j}^{(m-1)})\h_{kj}^{(m)}\right. \nonumber \\
&\quad\quad\left.+(\h_{ki}^{(m-1)}\cdot\h_{kj}^{(m-1)})(\h_{{\ell}j}^{(m)}+\h_{{\ell}i}^{(m-1)}\right).
\end{align}
\end{subequations}
\end{lemma}
\begin{proof}
Using the expression in \eqref{eq:3-3} and identity \eqref{fundId}, we have
\begin{align*}
w_{ij}(f^{(m)})
&=-\frac{1}{36}\sum_{\substack{\tau\in\mathbb{T}(\mathcal{M})\\ [\bv_i,\bv_j]\cup[\bv_k, \bv_{\ell}]\subset\tau\\ [\bv_i,\bv_j]\cap[\bv_k, \bv_{\ell}]=\emptyset}}
\frac{(\h_{k i}^{(m)}\times \h_{\ell i}^{(m)} )^{\top} (\h_{\ell j}^{(m)}  \times \h_{k j}^{(m)})}{\mu(\tau)} \nonumber \\
&=- \sum_{\substack{\tau\in\mathbb{T}(\mathcal{M})\\ [\bv_i,\bv_j]\cup[\bv_k, \bv_{\ell}]\subset\tau\\ [\bv_i,\bv_j]\cap[\bv_k, \bv_{\ell}]=\emptyset}}
\frac{(\h_{k i}^{(m)} \cdot \h_{\ell j}^{(m)} ) ( \h_{\ell i}^{(m)} \cdot \h_{k j}^{(m)} ) - (\h_{k i}^{(m)} \cdot \h_{k j}^{(m)} ) ( \h_{\ell i}^{(m)} \cdot \h_{\ell j}^{(m)} )}{36 \mu(\tau)} \label{eq:1st_term_wij}. 
\end{align*}
Hence, we can rewrite the numerator of $w_{ij}(f^{(m)})-w_{ij}(f^{(m-1)})$ as the following two terms:
\begin{align*}
& (\h_{k i}^{(m)} \cdot \h_{\ell j}^{(m)} ) ( \h_{\ell i}^{(m)} \cdot \h_{k j}^{(m)} ) - (\h_{k i}^{(m-1)} \cdot \h_{\ell j}^{(m-1)} ) ( \h_{\ell i}^{(m-1)} \cdot \h_{k j}^{(m-1)} ) \nonumber \\
=&(\h_{k i}^{(m)} \cdot \h_{\ell j}^{(m)} ) ( \h_{\ell i}^{(m)} \cdot \h_{k j}^{(m)} ) - (\h_{k i}^{(m-1)} \cdot \h_{\ell j}^{(m-1)} ) ( \h_{\ell i}^{(m-1)} \cdot \h_{k j}^{(m-1)} ) \\
&-(\h_{k i}^{(m-1)} \cdot \h_{\ell j}^{(m-1)} ) (\h_{\ell i}^{(m)} \cdot \h_{k j}^{(m)} ) + (\h_{k i}^{(m-1)} \cdot \h_{\ell j}^{(m-1)} ) (\h_{\ell i}^{(m)} \cdot \h_{k j}^{(m)} )\\
=&\left\{(\beps^{(m)}_k-\beps^{(m)}_i) \cdot \h_{\ell j}^{(m)} + \h_{k i}^{(m-1)} \cdot (\beps^{(m)}_\ell-\beps^{(m)}_j)\right\}(\h_{\ell i}^{(m)} \cdot \h_{k j}^{(m)})\\
&+\left\{(\beps^{(m)}_\ell-\beps^{(m)}_i) \cdot \h_{k j}^{(m)}+\h_{\ell i}^{(m-1)}\cdot(\beps^{(m)}_k-\beps^{(m)}_j)\right\}(\h_{k i}^{(m-1)} \cdot \h_{\ell j}^{(m-1)}) 
\end{align*}
and
\begin{align*}
& (\h_{k i}^{(m)} \cdot \h_{k j}^{(m)} ) ( \h_{\ell i}^{(m)} \cdot \h_{\ell j}^{(m)} ) - (\h_{k i}^{(m-1)} \cdot \h_{k j}^{(m-1)} ) ( \h_{\ell i}^{(m-1)} \cdot \h_{\ell j}^{(m-1)} ) \nonumber \\ 
=&(\h_{k i}^{(m)} \cdot \h_{k j}^{(m)} ) ( \h_{\ell i}^{(m)} \cdot \h_{\ell j}^{(m)} ) - (\h_{k i}^{(m-1)} \cdot \h_{k j}^{(m-1)} ) ( \h_{\ell i}^{(m-1)} \cdot \h_{\ell j}^{(m-1)} ) \\
&- (\h_{k i}^{(m-1)} \cdot \h_{k j}^{(m-1)} ) ( \h_{\ell i}^{(m)} \cdot \h_{\ell j}^{(m)} ) + (\h_{k i}^{(m-1)} \cdot \h_{k j}^{(m-1)} ) ( \h_{\ell i}^{(m)} \cdot \h_{\ell j}^{(m)} ) \\ 
=&\left\{(\beps^{(m)}_k-\beps^{(m)}_i) \cdot \h_{k j}^{(m)} + \h_{k i}^{(m-1)} \cdot (\beps^{(m)}_k-\beps^{(m)}_j)\right\}( \h_{\ell i}^{(m)} \cdot \h_{\ell j}^{(m)} )\\
&+\left\{(\beps^{(m)}_\ell-\beps^{(m)}_i) \cdot \h_{\ell j}^{(m)} +\h_{\ell i}^{(m-1)} \cdot (\beps^{(m)}_\ell-\beps^{(m)}_j)\right\}( \h_{k i}^{(m-1)} \cdot \h_{k j}^{(m-1)} ). 
\end{align*}
Combining the above results, we obtain
\begin{align*}
\widehat{w}_{ij}^{(m)} = \sum_{\substack{\tau\in\mathbb{T}(\mathcal{M})\\ [\bv_i,\bv_j]\cup[\bv_k, \bv_{\ell}]\subset\tau\\ [\bv_i,\bv_j]\cap[\bv_k, \bv_{\ell}]=\emptyset}}\left(-\c_i^{(m)} \cdot \beps_i^{(m)}-\c_j^{(m)} \cdot \beps_j^{(m)}+\c_k^{(m)} \cdot \beps_k^{(m)}+\c_{\ell}^{(m)} \cdot \beps_{\ell}^{(m)}\right),
\end{align*}
where $\c_i^{(m)}$, $\c_j^{(m)}$, $\c_k^{(m)}$ and $\c_{\ell}^{(m)}$ are defined in \eqref{eq:c_ijkl}.
\end{proof}

By the definition of $L_V(f)$ in \eqref{eq:2} and using the result in Lemma~\ref{lem:w_ij}, we obtain
{\small\begin{align}
& \left( \left( \begin{bmatrix} 
L_{\I}^{(m)} & L_{\B}^{(m)} 
\end{bmatrix} - \begin{bmatrix} 
L_{\I}^{(m-1)} & L_{\B}^{(m-1)} 
\end{bmatrix} \right) \mathbf{g}^{(m)}_s \right)_i  
= \sum_{[\bv_i,\bv_j] \in \mathbb{E}(\M) } \widehat{w}_{ij}^{(m)} (g_{sj}^{(m)} - g_{si}^{(m)}) \nonumber \\ 
=& \sum_{[\bv_i,\bv_j] \in \mathbb{E}(\M) } (g_{sj}^{(m)} - g_{si}^{(m)}) \sum_{\substack{\tau\in\mathbb{T}(\mathcal{M})\\ [\bv_i,\bv_j]\cup[\bv_k, \bv_{\ell}]\subset\tau\\ [\bv_i,\bv_j]\cap[\bv_k, \bv_{\ell}]=\emptyset}} \left(-\c_i^{(m)}\cdot \beps_i^{(m)}-\c_j^{(m)} \cdot \beps_j^{(m)}+\c_k^{(m)} \cdot \beps_k^{(m)}+\c_{\ell}^{(m)} \cdot \beps_{\ell}^{(m)}\right)  \nonumber \\
\equiv & \left( T_s^{(m)} \beps^{(m)} 
\right)_i \label{eq:w_ij_gs}
\end{align}}
for $i = 1, \ldots, n$ and $s = 1, 2 ,3$ with $T_s^{(m)} \in \mathbb{R}^{n \times 3n}$ and 
\begin{align}
\beps^{(m)} =  \begin{bmatrix}
 \beps^{(m)}_1 \\ \vdots \\ \beps^{(m)}_n
\end{bmatrix} \in \mathbb{R}^{3n}. \label{eq:beps_m}
\end{align}
With this critical result, we get the relationship between $\beps^{(m+1)}$ and $\beps^{(m)}$ as the following theorem.

\begin{theorem}
Let $\mathbf{g}^{(m)}_s$, $\beps^{(m)}$, and $T_s^{(m)}$, for $s = 1, 2, 3$, be defined in \eqref{eq:g_m}, \eqref{eq:beps_m}, and \eqref{eq:w_ij_gs}, respectively. 
Let
\begin{align*}
P = \begin{bmatrix}
\mathbf{e}_1 & \mathbf{e}_{n+1} & \mathbf{e}_{2n+1} & \mathbf{e}_{2} & \mathbf{e}_{n+2} & \mathbf{e}_{2n+2} & \cdots \mathbf{e}_{n} & \mathbf{e}_{2n} & \mathbf{e}_{3n}
\end{bmatrix}^{\top}
\end{align*}
be an $3n \times 3n$ permutation matrix.
Then
\begin{align}
\beps^{(m+1)}  = P \left( I_3 \otimes (L_{\I}^{(m)})^{-1} \right) \begin{bmatrix}
T_1^{(m)} \\ T_2^{(m)} \\ T_3^{(m)}
\end{bmatrix} \beps^{(m)} \equiv \mathcal{T}^{(m)} \beps^{(m)}. \label{eq:iter_eps}
\end{align}
\end{theorem}
\begin{proof}
From the definitions of $\beps^{s(m+1)}$, $\beps^{(m)}_t$ and $\beps^{(m)}$ in \eqref{eq:error_eps_m}, \eqref{eq:def_eps_t}, and \eqref{eq:beps_m}, respectively, we have
\begin{align}
P \begin{bmatrix}
\beps^{1(m+1)} \\ \beps^{2(m+1)} \\ \beps^{3(m+1)}
\end{bmatrix} = \begin{bmatrix}
\beps^{(m+1)}_1 \\ \vdots \\ \beps^{(m+1)}_n
\end{bmatrix} = \beps^{(m+1)}. \label{eq:reorder_eps}
\end{align}
The results in \eqref{eq:eps_form}, \eqref{eq:w_ij_gs}, \eqref{eq:beps_m} and \eqref{eq:reorder_eps} imply that
\begin{align}
\beps^{(m+1)} &= P \begin{bmatrix}
\left( L_{\I}^{(m)} \right)^{-1} \left( \begin{bmatrix} 
L_{\I}^{(m)} & L_{\B}^{(m)} 
\end{bmatrix} - \begin{bmatrix} 
L_{\I}^{(m-1)} & L_{\B}^{(m-1)} 
\end{bmatrix} \right) \mathbf{g}^{(m)}_1 \\
\left( L_{\I}^{(m)} \right)^{-1} \left( \begin{bmatrix} 
L_{\I}^{(m)} & L_{\B}^{(m)} 
\end{bmatrix} - \begin{bmatrix} 
L_{\I}^{(m-1)} & L_{\B}^{(m-1)} 
\end{bmatrix} \right) \mathbf{g}^{(m)}_2 \\
\left( L_{\I}^{(m)} \right)^{-1} \left( \begin{bmatrix} 
L_{\I}^{(m)} & L_{\B}^{(m)} 
\end{bmatrix} - \begin{bmatrix} 
L_{\I}^{(m-1)} & L_{\B}^{(m-1)} 
\end{bmatrix} \right) \mathbf{g}^{(m)}_3
\end{bmatrix} \nonumber \\
&= P \left( I_3 \otimes  ( L_{\I}^{(m)}  )^{-1} \right)
\begin{bmatrix}
T_1^{(m)} \\ T_2^{(m)} \\ T_3^{(m)}
\end{bmatrix} \beps^{(m)}. \label{eq:eps_mp1}
\end{align}
\end{proof}

Like the analysis of convergence for the stretch energy minimization for equiareal parameterizations in \cite{HuLL22}, in the following theorem we show that Algorithm~\ref{alg:VSEM} for the volume-preserving parameterization is R-linearly convergent.

\begin{theorem}
\label{thm3.6}
Let $\mathcal{M}\subset\mathbb{R}^3$ be a simply connected 3-manifold with a single genus-zero boundary. Let $\mathbf{f}_{\mathtt{I}}^{s(m)}$ for $s=1, 2, 3$ be defined in \eqref{eq:f_Is} and $\beps^{(m)}$ with $\beps_t^{(m)}$ being defined in \eqref{eq:reorder_eps} and \eqref{eq:def_eps_t}, respectively. Define
\begin{align}
\mathcal{P}_m = (\mathcal{T}^{(m)} \cdots \mathcal{T}^{(0)}) \label{eq:mtx_Pm}
\end{align}
with $\mathcal{T}^{(i)}$ in \eqref{eq:iter_eps}. Considering Figure \ref{fig:eigs_svds}, we suppose that there exists $m_0 > 0$ and $c > 0$ such that
\begin{align}
\| \mathcal{P}_m \|_2 \leq c, \quad \rho(\mathcal{P}_{m+1}) \leq \rho(\mathcal{P}_m) < 1, \label{eq:rho_P}
\end{align}
and $\mathcal{P}_m$ is diagonalizable for $m \geq m_0$.
Then, there exists a sequence $\{\mathbf{f}_{\mathtt{I}}^{s(m_j)}\}$ for $s=1, 2, 3$ satisfying R-linear convergence, i.e.,
\begin{align*}
\|\beps^{(m_j)} \|_{\infty}^{1/m_j} <1.
\end{align*}
\end{theorem}
\begin{proof}
Since $\mathcal{P}_m$ is diagonalizable, for $m \geq m_0$, from the assumption \eqref{eq:rho_P}, there exists a subsequence $\{m_j\}$ such that
\begin{align*}
\lim_{j \to \infty} \mathcal{P}_{m_j}^{1/m_j} = \mathcal{A}
\end{align*}
with $\rho(\mathcal{A}) < 1$. Here, $\mathcal{P}_{m_j}^{1/m_j}$ is the $m_j$-root of $\mathcal{P}_{m_j}$ for which the eigenvalues are taken as the principal values of the eigenvalues of $\mathcal{P}_{m_j}$.
Let
\begin{align*}
\mathcal{P}_{m_j}^{1/m_j} = \mathcal{A} + \mathcal{E}_{m_j}
\end{align*}
with $\mathcal{E}_{m_j} \to 0$ as $j \to \infty$.
There is an operator norm $\|\cdot\|_{\ast}$ and a constant $M_\infty>0$ such that $\| \mathcal{A} \|_{\ast} < 1$ and $\| \mathcal{A} \|_{\infty} \leq M_{\infty} \| \mathcal{A} \|_{\ast}$. Then, for $\beps^{(m_j)} = \mathcal{P}^{(m_j)} \beps^{(0)} = (\mathcal{T}^{(m_j)} \cdots \mathcal{T}^{(0)}) \beps^{(0)}$, we have
\begin{align*}
\left\|\beps^{(m_j)} \right\|^{1/m_j}_{\infty}
&\leq \left\| \mathcal{P}_{m_j} \right\|^{1/m_j}_{\infty} \| \beps^{(0)} \|^{1/m_j}_{\infty} = \| (\mathcal{A} + E_{m_j})^{m_j} \|^{1/m_j}_{\infty} \| \beps^{(0)} \|^{1/m_j}_{\infty} \\
&\leq \| (\mathcal{A} + E_{m_j})^{m_j} \|^{1/m_j}_{\ast} M_{\infty}^{1/m_j} \| \beps^{(0)} \|^{1/m_j}_{\infty} \\
&\leq \| (\mathcal{A} + E_{m_j})  \|_{\ast} M_{\infty}^{1/m_j} \| \beps^{(0)} \|^{1/m_j}_{\infty} <1
\end{align*}
for $j$ sufficiently large.
\end{proof}

\section{Projected gradient method for VOMT maps}
\label{sec:5}

The volume-/mass-preserving parameterization computed by Algorithm \ref{alg:VSEM} is not unique. For example, given a volume-/mass-preserving parameterization $\f$, consider the rotation $R\in SO(3) = \{R \in \mathbb{R}^{3 \times 3}| R^{\top} R = I_3, \mbox{det}(R) = 1 \}$. The map $\f R^\top$ is also a volume-/mass-preserving parameterization and hence a minimizer of the volumetric stretch energy \eqref{eq:energy_fun}. In practical applications, it is usually desired that the map be unique.
To guarantee the uniqueness of the parameterization, it is natural to impose a constraint to minimize the displacement of each point with the volume being its weight. Such a map is called a VOMT map.
More precisely, the VOMT map minimizes the cost function
\begin{equation*}
\mathcal{C}(f) = \int_\mathcal{M} \| \bv - f(\bv) \|_2^2 \,\nu(\bv)
\end{equation*}
under the constraint that $f:\mathcal{M}\to\mathbb{B}^3$ is volume-/mass-preserving with respect to the volume/mass measure $\nu:\mathcal{M}\to\mathbb{R}_+$.
The VSEM algorithm, Algorithm \ref{alg:VSEM}, can be applied as a projection $\Pi_{\mathscr{F}}:\mathbb{R}^{n\times 3}\to\mathbb{R}^{n\times 3}$ to the space
$$
\mathscr{F} = \left\{ \f R^\top \in\mathbb{R}^{n\times 3} \left| 
\begin{aligned}
&\textrm{ $\f$ represents a volume-/mass-preserving map $f:\mathcal{M}\to\mathbb{B}^3$} \\
&\textrm{ $R = \argmin_{R\in SO(3)} \mathcal{C}(\f R^\top)$ is the optimal rotation with respect to $\mathcal{C}$}
\end{aligned} \right.\right\}
$$
so that the constraint can be conveniently satisfied.

When $f$ is a piecewise affine map, the volume measure is piecewise constant so that the cost function can be formulated as
\begin{equation} \label{eq:cost}
\mathcal{C}(f) = \sum_{\bv\in\mathbb{V}(\mathcal{M})} \| \bv - f(\bv) \|_2^2 \,\nu(\bv),
\end{equation}
where
$$
\nu(\bv) = \frac{1}{4} \sum_{\tau\in N(\bv)} \mu(\tau)
$$
with $N(\bv)$ being the set of neighboring tetrahedra of $\bv$.
Noting that the piecewise affine map $f$ is represented as a matrix $\mathbf{f}$ as in \eqref{eq:f}, the cost function $\mathcal{C}$ in \eqref{eq:cost} can be written as
\begin{align*}
\mathcal{C}(\f) \equiv \mathcal{C}((\mathbf{f}^1)^\top, (\mathbf{f}^2)^\top, (\mathbf{f}^3)^\top) 
= \sum_{t=1}^n \sum_{s=1}^3 (v_t^s - f_t^s)^2 \,\nu(\bv_t).
\end{align*}
The gradient of $\mathcal{C}$ is written as
$$
\nabla\mathcal{C}(\f) = -2 (I_3 \otimes \diag(\nu(\bv_1), \cdots, \nu(\bv_n))) \,
\vec(\bv-\f),
$$
where $\bv = [\bv^1, \bv^2, \bv^3]$ with $\bv^s = [v_1^s, \ldots, v_n^s]^\top$, $s=1,2,3$, and
$
\vec(\f) \equiv [
(\f^1)^{\top},
(\f^2)^{\top},
(\f^3)^{\top}
]^{\top}.
$

The projected gradient method for VOMT maps is performed as follows.
First, we compute a volume-/mass-preserving map ${\f}^{(0)}$ by Algorithm \ref{alg:VSEM} with the boundary being an area-preserving OMT map computed by the AOMT algorithm in \cite{YuHL21}.
Then, the map is updated along the negative gradient direction as
\begin{subequations} \label{eq:grad}
\begin{equation} 
\vec(\overline{\f}^{(m+1)}) = \vec(\f^{(m)}) - \alpha_m \nabla\mathcal{C}(\f^{(m)}),
\end{equation}
where the step size $\alpha_m$ is computed by 
\begin{align}
\alpha_m = \argmin_{\alpha>0} \mathcal{C}(\overline{\f}^{(m+1)}).
\end{align}
\end{subequations}
Next, the map $\overline{\f}^{(m+1)}$ is projected to $\mathscr{F}$ by the VSEM algorithm with the initial Laplacian matrix being $L_V(\overline{\f}^{(m+1)})$ and the boundary map being that of $\f^{(0)}$.
As a result, the updated map is written as
\begin{equation} \label{eq:fk}
\f^{(m+1)} = \Pi_\mathscr{F}(\overline{\f}^{(m+1)}).
\end{equation}
The iteration terminates when $\mathcal{C}(\f^{(m+1)}) \geq \mathcal{C}(\f^{(m)})$.

The details of the computational procedure are summarized in Algorithm \ref{alg:VOMT}.

\begin{algorithm}
\caption{Projected gradient method for the VOMT map} \label{alg:VOMT}
\begin{algorithmic}[1]
\Require A simply connected tetrahedral mesh $\mathcal{M}$ and a tolerance $\varepsilon$.
\Ensure A VOMT map $f:\mathcal{M}\to\mathbb{B}^3$ induced by $\mathbf{f}$.
\State Let $n$ be the number of vertices of $\mathcal{M}$.
\State Let $\mathtt{B} = \{t \,|\, v_t\in\partial\mathcal{M}\}$ and $\mathtt{I} = \{1, \ldots, n\} \backslash\mathtt{B}$.
\State Compute the initial map $\f$ by Algorithm \ref{alg:VSEM}.
\State Let $\delta\gets\infty$ and $\widehat{\mathbf{f}}\gets\mathbf{f}$ ($\widehat{f}\gets f$, where $\widehat{f}$ is induced by $\widehat{\mathbf{f}}$).
\While{$\delta>\varepsilon$}
\State Update $\f$ as in \eqref{eq:grad} and \eqref{eq:fk}.
\State Update $\delta \gets \mathcal{C}(\widehat{f})-\mathcal{C}(f)$ and $\widehat{\mathbf{f}}\gets \mathbf{f}$ ($\widehat{f}\gets f$).
\EndWhile
\State \Return $\mathbf{f}$.
\end{algorithmic}
\end{algorithm}

\subsection{Convergence of the projected gradient method}

Now, we provide a rigorous proof for the convergence of Algorithm~\ref{alg:VOMT}.

\begin{theorem} \label{thm:ProjGrad}
Let $\mathcal{C}$ be convex and $L$-smooth, i.e., $\nabla^2 \mathcal{C} - L I$ is positive semidefinite.
Assume the projection $\Pi_\mathscr{F}$ satisfies the properties of projection
\begin{subequations}
\begin{align} 
\left( \vec( \Pi_\mathscr{F}(\overline{\f}^{(m)}) ) - \vec( \overline{\f}^{(m)} ) \right)^\top \left(  \vec( \Pi_\mathscr{F}(\overline{\f}^{(m)}) ) - \vec( \f^{(m)} ) \right) \leq 0, \label{eq:assumption2} \\
\left( \vec( \Pi_\mathscr{F}(\overline{\f}^{(m)}) ) - \vec( \overline{\f}^{(m)} ) \right)^\top \left(  \vec( \Pi_\mathscr{F}(\overline{\f}^{(m)}) ) - \vec( \f^{*} ) \right) \leq 0, \label{eq:assumption3}
\end{align}
\end{subequations}
and nonexpensiveness
\begin{equation} \label{eq:assumption1}
\left\|\vec( \Pi_\mathscr{F}(\overline{\f}^{(m)}) ) - \vec( \Pi_\mathscr{F}({\f}^{*}) )\right\|_2 \leq \left\|\vec( \overline{\f}^{(m)} ) - \vec( {\f}^{*} )\right\|_2,
\end{equation}
for every $m\geq 1$.
Suppose $\alpha_m \in [\eta,\frac{2}{L}]$ with $\eta>0$ for every $m\geq 1$.
Then,
\begin{align}
\mathcal{C}(\f^{(m)}) - \mathcal{C}(\f^{*}) \leq \frac{4\beta + \mathcal{C}(\f^{(0)}) - \mathcal{C}(\f^{*})}{m+1}, \label{eq:thm5.1}
\end{align}
where $\beta = \frac{\|\vec(\f^{(0)}) - \vec(\f^{*})\|_2^2}{2\eta+\eta^2 L}$ is a constant.
\end{theorem}
\begin{proof}
For convenience, let
\begin{alignat*}{2}
\delta^*\mathcal{C}(\f^{(m)}) &\equiv \mathcal{C}(\f^{(m)}) - \mathcal{C}(\f^{*}), &\quad
\delta\mathcal{C}(\f^{(m)}) &\equiv \mathcal{C}(\f^{(m)}) - \mathcal{C}(\f^{(m+1)}), \\
\delta^*\f^{(m)} &\equiv \vec(\f^{(m)}) - \vec(\f^{*}), & \quad
\delta\f^{(m)} &\equiv \vec(\f^{(m)}) - \vec(\f^{(m+1)}).
\end{alignat*}
Then, by \eqref{eq:assumption1} and the fundamental theorem of calculus, we have
\begin{align}
\left\| \delta^*\f^{(m+1)} \right\|_2 
&= \left\| \vec( \Pi_\mathscr{F}(\overline{\f}^{(m+1)}) ) - \vec( \Pi_{\mathscr{F}}(\f^{*}) ) \right\|_2 \nonumber \\
&\leq \left\|\vec( \overline{\f}^{(m+1)} ) - \vec( \f^{*} )\right\|_2 && \text{(by \eqref{eq:assumption1})} \nonumber \\
&= \left\|\vec(\f^{(m)}) - \alpha_m \nabla\mathcal{C}(\f^{(m)}) - \vec(\f^{*})\right\|_2 \nonumber \\
&= \left\|\left( I-\alpha_m \int_0^1 \nabla^2\mathcal{C}(\f^{*} + t(\f^{(m)} - \f^{*}))\,\mathrm{d}t \right) \delta^*\f^{(m)}\right\|_2 \nonumber\\
&\hspace{10mm} \text{(by the fundamental theorem of calculus)} \nonumber \\
& \leq \max_{t\in[0,1]} \left\| I-\alpha_m \nabla^2\mathcal{C}(\f^{*} + t(\f^{(m)} - \f^{*})) \right\|_2 \left\| \delta^*\f^{(m)} \right\|_2 \nonumber \\
& \leq \left\| \delta^*\f^{(m)} \right\|_2. \label{eq:d}
\end{align}
The last inequality follows from convex and $L$-smooth properties of $\mathcal{C}$ with $\alpha_m L \leq 2$.
However, for every $k\geq 1$, it holds that
\begin{align*}
& \frac{1}{\alpha_m} \left\| \delta\f^{(m)} \right\|_2^2 - \nabla\mathcal{C}(\f^{(m)})^\top \delta\f^{(m)}\\
&= \left( \frac{1}{\alpha_m}\delta\f^{(m)} - \nabla\mathcal{C}(\f^{(m)})\right)^\top \delta\f^{(m)} \\
&= \frac{1}{\alpha_m}\left( \vec(\f^{(m+1)}) - (\vec(\f^{(m)}) -\alpha_m\nabla\mathcal{C}(\f^{(m)}) ) \right)^\top \left( \vec(\f^{(m+1)}) - \vec(\f^{(m)}) \right) \\
&= \frac{1}{\alpha_m}\left( \vec( \Pi_\mathscr{F}(\overline{\f}^{(m)}) ) - \vec(\overline{\f}^{(m)}) \right)^\top \left( \vec( \Pi_\mathscr{F}(\overline{\f}^{(m)}) ) - \vec(\f^{(m)}) \right) \leq 0, \quad \text{(by \eqref{eq:assumption2})}
\end{align*}
which implies that
\begin{equation} \label{eq:result2}
\nabla\mathcal{C}(\f^{(m)})^\top \delta\f^{(m)} \geq \frac{1}{\alpha_m} \left\| \delta\f^{(m)} \right\|_2^2.
\end{equation}
Similarly, by \eqref{eq:assumption3}, we have
\begin{equation} \label{eq:result3}
\nabla\mathcal{C}(\f^{(m)})^\top \delta^*\f^{(m+1)} \leq \frac{1}{\alpha_m} (\delta\f^{(m)})^\top\delta^*\f^{(m+1)}.
\end{equation}
Since $\mathcal{C}$ is convex, we obtain
\begin{equation} \label{eq:convex}
\delta^*\mathcal{C}(\f^{(m)}) \leq \nabla\mathcal{C}(\f^{(m)}) ^\top \delta^*\f^{(m)}. 
\end{equation}
Since $\mathcal{C}$ is $L$-smooth, we have
\begin{equation} \label{eq:Lsmooth}
\delta\mathcal{C}(\f^{(m)}) \geq \nabla\mathcal{C}(\f^{(m)})^\top \delta\f^{(m)} + \frac{L}{2} \left\| \delta\f^{(m)} \right\|_2^2.
\end{equation}
Then, by using \eqref{eq:result3}, \eqref{eq:convex} and \eqref{eq:Lsmooth}, it holds that
\begin{align}
\delta^*\mathcal{C}(\f^{(m+1)})
&= \delta^*\mathcal{C}(\f^{(m)}) -\delta\mathcal{C}(\f^{(m)}) \nonumber \\
&\leq \nabla\mathcal{C}(\f^{(m)})^\top \delta^*\f^{(m)} - \left( \nabla\mathcal{C}(\f^{(m)})^\top \delta\f^{(m)} + \frac{L}{2} \left\| \delta\f^{(m)} \right\|_2^2 \right) &&\text{(by \eqref{eq:convex}, \eqref{eq:Lsmooth})} \nonumber \\
&= \nabla\mathcal{C}(\f^{(m)})^\top \delta^*\f^{(m+1)} - \frac{L}{2} \left\| \delta\f^{(m)} \right\|_2^2 \nonumber \\
&\leq \frac{1}{\alpha_m} (\delta\f^{(m)})^\top \delta^*\f^{(m+1)} - \frac{L}{2} \left\| \delta\f^{(m)} \right\|_2^2 && \text{(by \eqref{eq:result3})} \nonumber \\
&= \frac{1}{\alpha_m} (\delta\f^{(m)})^\top \left( \delta^*\f^{(m)} -\delta\f^{(m)} \right) - \frac{L}{2} \left\| \delta\f^{(m)} \right\|_2^2 \nonumber \\
&= \frac{1}{\alpha_m} (\delta\f^{(m)})^\top \delta^*\f^{(m)} - \left(\frac{1}{\alpha_m} + \frac{L}{2}\right) \left\| \delta\f^{(m)} \right\|_2^2 \nonumber \\
& \leq \frac{1}{\alpha_m} (\delta\f^{(m)})^\top \delta^*\f^{(m)} 
\leq \frac{1}{\alpha_m} \left\| \delta\f^{(m)} \right\|_2  \left\| \delta^*\f^{(m)} \right\|_2. \label{pf:thm5.1_14a}
\end{align}
From \eqref{eq:d} and \eqref{pf:thm5.1_14a}, we have
\begin{align}
\left\| \delta\f^{(m)} \right\|_2 \geq \frac{\alpha_m \delta^*\mathcal{C}(\f^{(m+1)})}{\| \delta^*\f^{(m)} \|_2} \geq \frac{\alpha_m \delta^*\mathcal{C}(\f^{(m+1)})}{\| \delta^*\f^{(0)} \|_2}. \label{pf:thm5.1_14b}
\end{align}
Using the assumption $\alpha_m \in [\eta,\frac{2}{L}]$ and the results in \eqref{eq:result2}, \eqref{eq:Lsmooth}, and \eqref{pf:thm5.1_14b}, it holds that
\begin{align*}
&\delta^*\mathcal{C}(\f^{(m)}) - \delta^*\mathcal{C}(\f^{(m+1)}) 
= \delta\mathcal{C}(\f^{(m)}) \\
&\geq \nabla\mathcal{C}(\f^{(m)})^\top \delta\f^{(m)} + \frac{L}{2} \left\| \delta\f^{(m)} \right\|_2^2 && \text{(by \eqref{eq:Lsmooth})} \\
&\geq \frac{1}{\alpha_m} \left\| \delta\f^{(m)} \right\|_2^2 + \frac{L}{2} \left\| \delta\f^{(m)} \right\|_2^2 = \left( \frac{1}{\alpha_m} + \frac{L}{2} \right) \left\| \delta\f^{(m)} \right\|_2^2  && \text{(by \eqref{eq:result2})}\\
&
\geq \frac{\alpha_m^2}{\| \delta^*\f^{(0)} \|_2^2} \left( \frac{1}{\alpha_m} + \frac{L}{2} \right) \delta^*\mathcal{C}(\f^{(m+1)})^2 
= \frac{2\alpha_m+\alpha_m^2 L}{2\|\delta^*\f^{(0)}\|_2^2} \delta^*\mathcal{C}(\f^{(m+1)})^2 && \text{(by \eqref{pf:thm5.1_14b})} \\
&\geq \frac{2\eta + \eta^2 L}{2\|\delta^*\f^{(0)}\|_2^2} \delta^*\mathcal{C}(\f^{(m+1)})^2 
= \frac{1}{2\beta} \delta^*\mathcal{C}(\f^{(m+1)})^2,  
\end{align*}
where $\beta = \frac{\|\delta^*\f^{(0)}\|_2^2}{2\eta+\eta^2 L}$,
which implies that
\begin{align*}
\delta^*\mathcal{C}(\f^{(m+1)})^2 + 2\beta \delta^*\mathcal{C}(\f^{(m+1)}) - 2\beta \delta^*\mathcal{C}(\f^{(m)}) \leq 0.
\end{align*}
This means that
\begin{align}
0\leq \delta^*\mathcal{C}(\f^{(m+1)}) \leq \sqrt{\beta^2 + 2\beta \delta^*\mathcal{C}(\f^{(m)})} - \beta. \label{pf:thm5.1_14c}
\end{align}
Now, we use the result in \eqref{pf:thm5.1_14c} and mathematical induction on $m$ to prove \eqref{eq:thm5.1}. First, we derive the following result, which will be used in the mathematical induction step:
\begin{align*}
m\geq 1 
&\iff 4m\geq 4 
\iff 8m \geq 4m+4 
\iff m^2+8m \geq (m+2)^2 \\
&\iff \sqrt{m^2+8m} \geq m+2
\iff \sqrt{m^2+8m} + m \geq 2(m+1).
\end{align*}
By induction on $m$, we suppose
\begin{align}
\delta^*\mathcal{C}(\f^{(m-1)}) \leq \frac{4\beta + \delta^*\mathcal{C}(\f^{(0)})}{m}. \label{eq:induction_m}
\end{align}
Substituting \eqref{eq:induction_m} into \eqref{pf:thm5.1_14c}, we have
\begin{align*}
\delta^*\mathcal{C}(\f^{(m)})
&\leq \sqrt{\beta^2 + 2\beta \left(\frac{4\beta + \delta^*\mathcal{C}(\f^{(0)})}{m}\right)} - \beta \\
&= \frac{2\beta \left(\frac{4\beta + \delta^*\mathcal{C}(\f^{(0)})}{m}\right)}{\sqrt{\beta^2 + 2\beta \left(\frac{4\beta + \delta^*\mathcal{C}(\f^{(0)})}{m}\right)} + \beta} 
= \frac{2\beta \left(4\beta + \delta^*\mathcal{C}(\f^{(0)})\right)}{\sqrt{m^2\beta^2 + 2m\beta \left(4\beta + \delta^*\mathcal{C}(\f^{(0)})\right)} + m\beta} \\
&\leq \frac{2\beta \left(4\beta + \delta^*\mathcal{C}(\f^{(0)})\right)}{\sqrt{m^2\beta^2 + 2m\beta \left(4\beta \right)} + m\beta} 
= \frac{2 \left(4\beta + \delta^*\mathcal{C}(\f^{(0)})\right)}{\sqrt{m^2 + 8m} + m} \\
&\leq \frac{2 \left(4\beta + \delta^*\mathcal{C}(\f^{(0)})\right)}{2(m+1)} 
= \frac{4\beta + \delta^*\mathcal{C}(\f^{(0)})}{m+1}.
\end{align*}
\end{proof}

\subsection{Nesterov's accelerated gradient method}
Under mild assumptions of projection and nonexpensiveness properties, Theorem~\ref{thm:ProjGrad} shows that the Algorithm~\ref{alg:VOMT} has a convergence rate of $\mathcal{O}(1/m)$.
It is well known that the rate of convergence of the projected gradient method can be improved by applying Nesterov's accelerated gradient method \cite{nesterov1983} and its variation, FISTA \cite{BeTe09}. More precisely, the iteration \eqref{eq:grad} is replaced by
\begin{align}
\vec(\widehat{\f}^{(m+1)}) &= \vec(\f^{(m)}) - \alpha_m \nabla\mathcal{C}(\f^{(m)}), \nonumber \\
\vec(\overline{\f}^{(m+1)}) &= \vec(\widehat{\f}^{(m+1)}) + \frac{k}{k+3} \left(\vec(\widehat{\f}^{(m+1)}) - \vec(\widehat{\f}^{(m)}) \right), \label{eq:Nesterov}
\end{align}
for Nesterov's method, and
\begin{subequations} \label{eq:varNesterov}
\begin{align}
\vec(\widehat{\f}^{(m+1)}) &= \vec(\f^{(m)}) - \alpha_m \nabla\mathcal{C}(\f^{(m)}),   \\
\vec(\overline{\f}^{(m+1)}) &= (1-\beta_m) \vec(\widehat{\f}^{(m+1)}) + \beta_m \vec(\widehat{\f}^{(m)}), \\
\beta_m &= \frac{1-\lambda_m}{\lambda_{m+1}},~ \lambda_m = \frac{1+\sqrt{1+4\lambda_{m-1}^2}}{2},~ \lambda_0 = 0, 
\end{align}
\end{subequations}
for its variation, FISTA, respectively.
According to several numerical experiments, FISTA \eqref{eq:varNesterov} performs slightly better than the primitive Nesterov method \eqref{eq:Nesterov}. Therefore, we adopt \eqref{eq:varNesterov} to accelerate the projected gradient method in Algorithm \ref{alg:VOMT}.

\section{Numerical experiments}
\label{sec:6}

In this section, we first give numerical results to show the existence of the assumptions in \eqref{eq:rho_P} and then demonstrate the R-linear convergence of Algorithm~\ref{alg:VSEM} in Subsection \ref{subsec:6.1}. In Subsection~\ref{subsec:6.3}, numerical validation is used to present the $\frac{1}{m}$ convergence of Algorithm~\ref{alg:VOMT} in Theorem~\ref{thm:ProjGrad} and to show the accelerated effect of the FISTA accelerated gradient method. Finally, we demonstrate the numerical results from practical implementation in Subsection~\ref{subsec:prac_impl}.

Various benchmark tetrahedral mesh models demonstrated in Figure~\ref{fig:benchmark_mesh} are from Jacobson's GitHub~\cite{Jacobson}, Gu's website~\cite{GuOMT}, and the BraTS databases~\cite{BaGB21,BaAS17}.
Some tetrahedral mesh models are generated by using \texttt{iso2mesh} \cite{FaBo09,TrYF20} and \texttt{JIGSAW} mesh generators \cite{Engw15,Engw14,Engw16,EnIv14,EnIv16}.

\begin{figure}
\centering
\begin{tabular}{cccc}
\includegraphics[height=3.0cm]{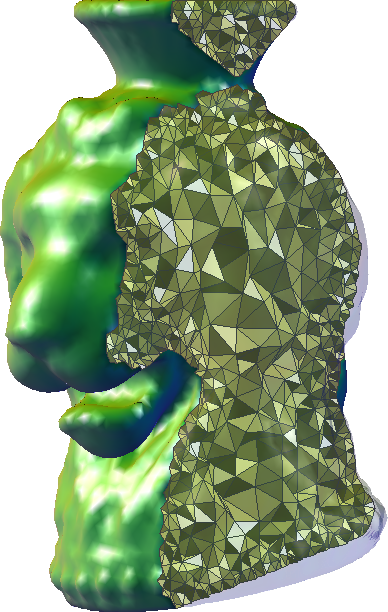} &
\includegraphics[height=3.0cm]{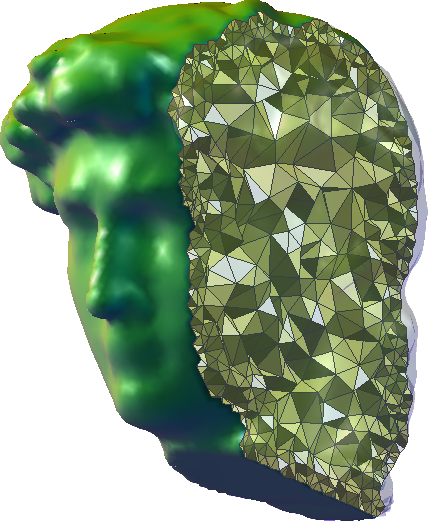} &
\includegraphics[height=3.0cm]{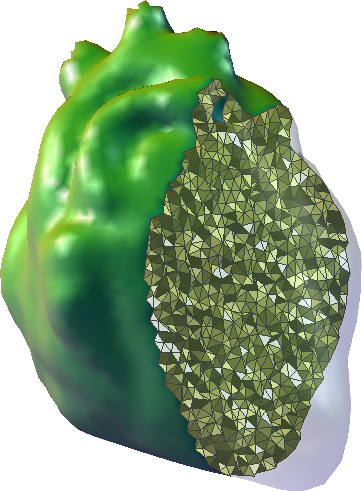} &
\includegraphics[height=3.0cm]{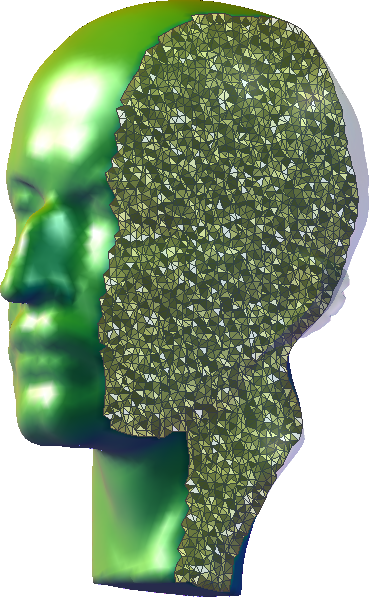} \\
(a) Lion & (b) David & (c) Heart & (d) Max Planck \\[0.5cm]
\includegraphics[height=2.7cm]{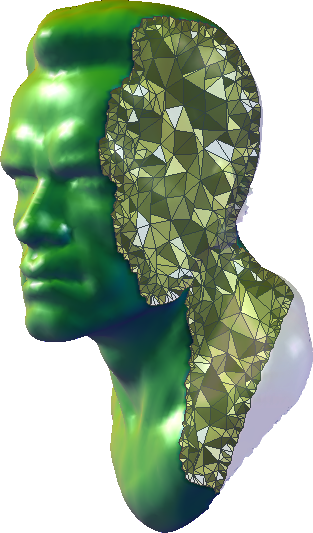} &
\includegraphics[height=2.7cm]{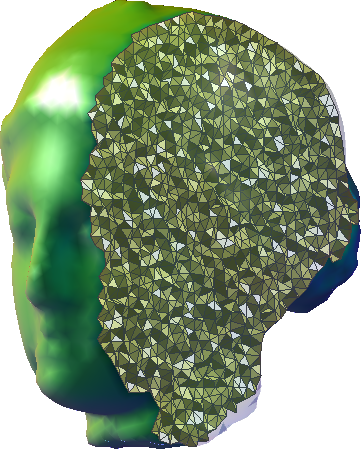} &
\includegraphics[height=2.7cm]{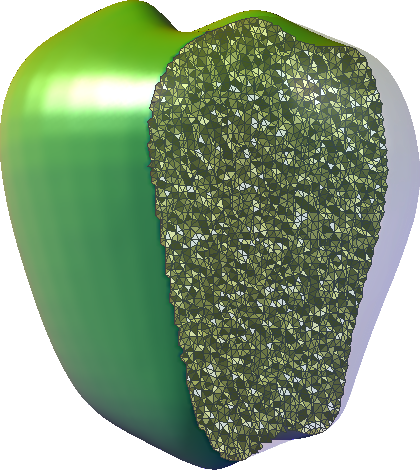} &
\includegraphics[height=2.5cm]{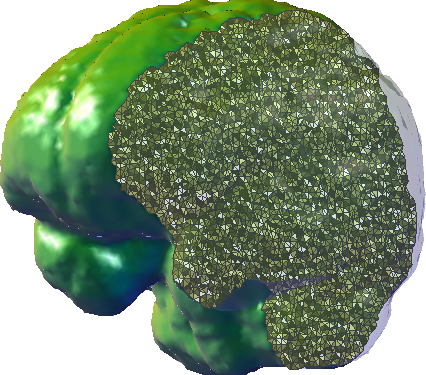} \\
(e) Arnold & (f) Igea & (g) Apple & (h) Brain
\end{tabular}
\caption{The benchmark tetrahedral mesh models.}
\label{fig:benchmark_mesh}
\end{figure}

\subsection{R-linear convergence of the VSEM algorithm}
\label{subsec:6.1}

\begin{figure}
\center
\begin{subfigure}[b]{0.32\textwidth}
\center
\includegraphics[width=\textwidth]{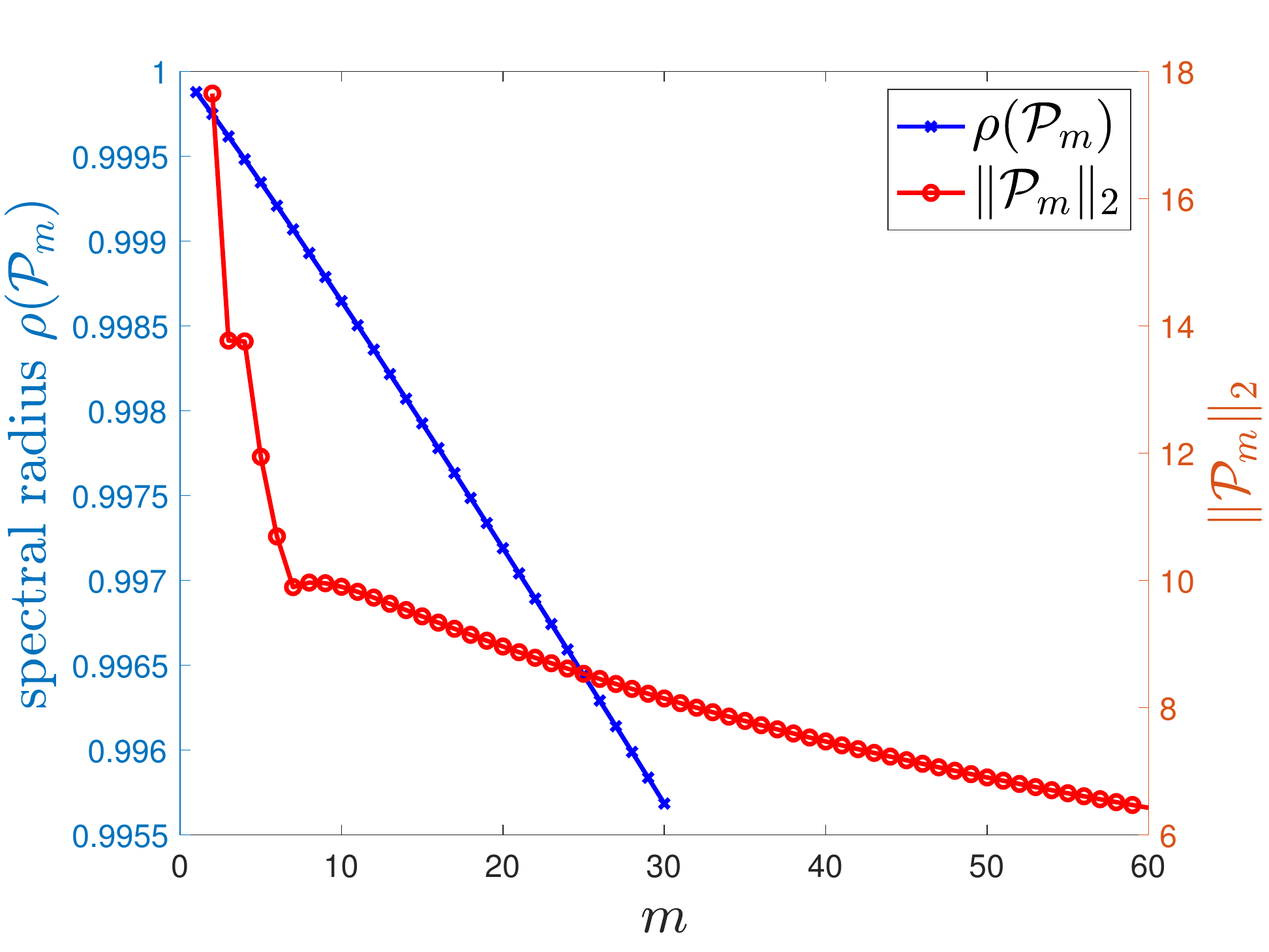}
\caption{Apple}
\label{fig:Apple_eigs_svds}
\end{subfigure}
\begin{subfigure}[b]{0.32\textwidth}
\center
\includegraphics[width=\textwidth]{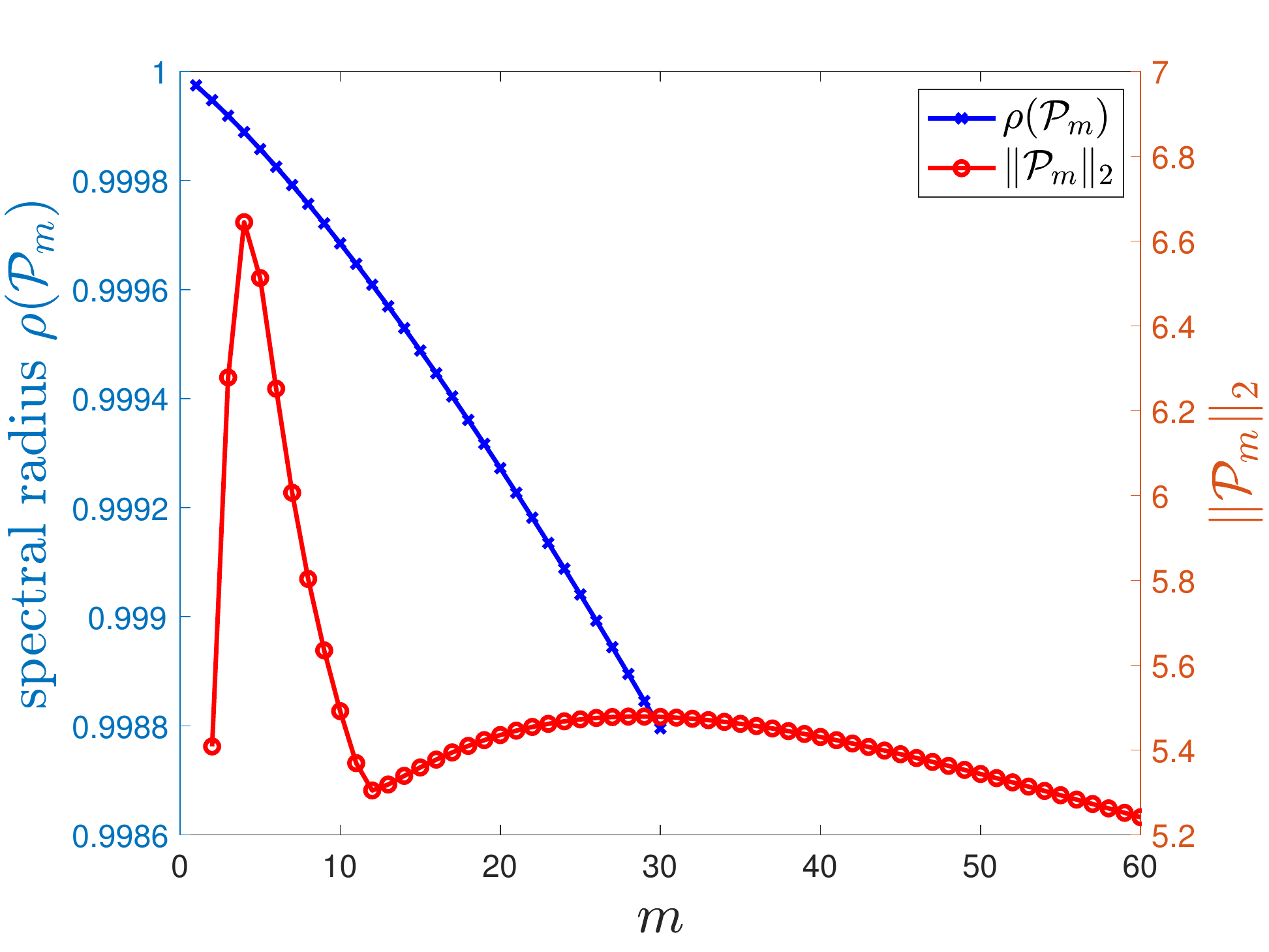}
\caption{Brain}
\label{fig:Brain_eigs_svds}
\end{subfigure}
\begin{subfigure}[b]{0.32\textwidth}
\center
\includegraphics[width=\textwidth]{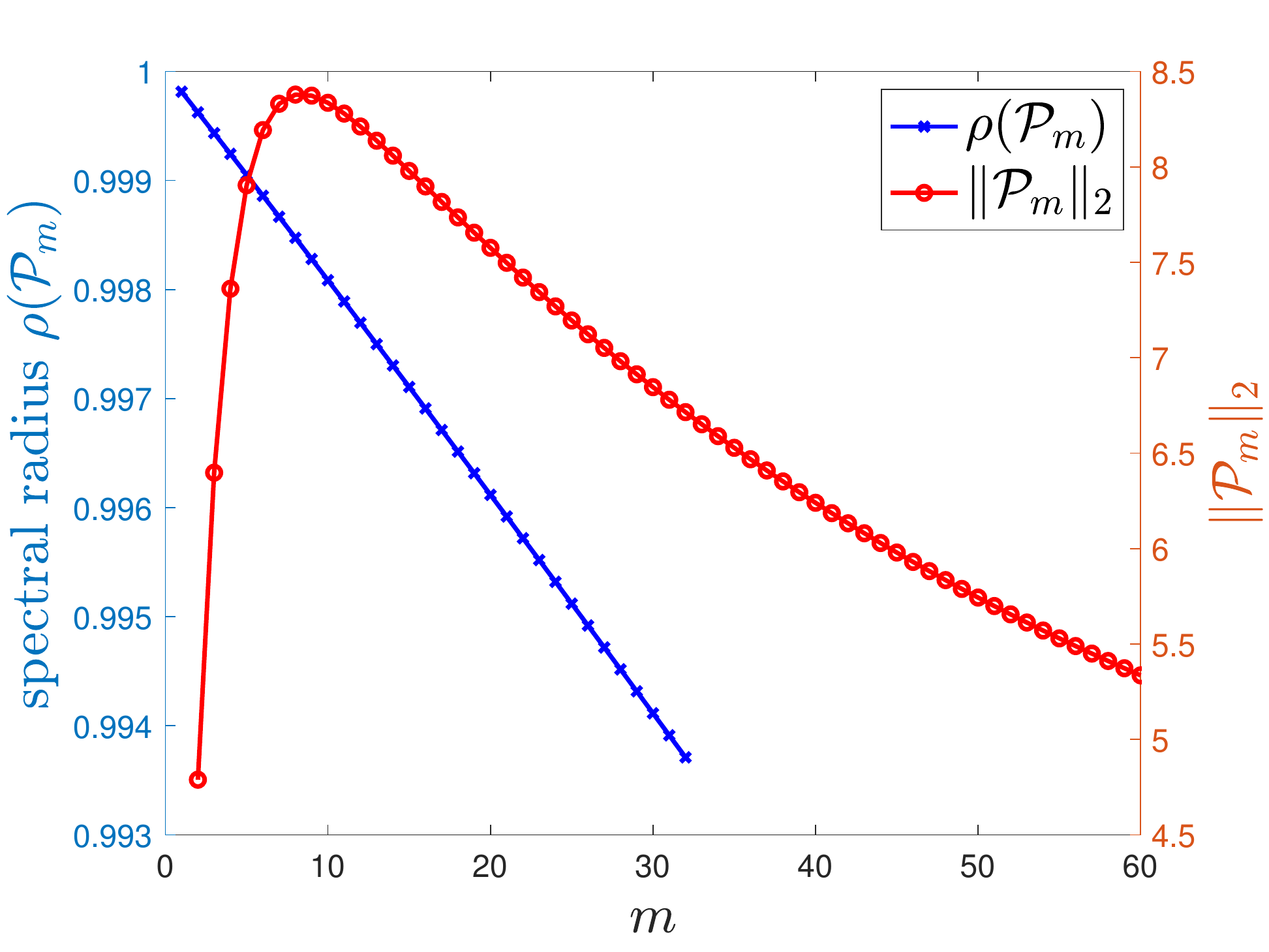}
\caption{Venus}
\label{fig:Venus_eigs_svds}
\end{subfigure}
\caption{Spectral radius of $\mathcal{P}_m$ and $\| \mathcal{P}_m \|_2$ for the Apple, Brain and Venus benchmark problems.}
\label{fig:eigs_svds}
\end{figure}

To show the R-linear convergence of the VSEM in Theorem~\ref{thm3.6}, we assume that the matrix $\mathcal{P}_m$ in \eqref{eq:mtx_Pm} satisfies the conditions in \eqref{eq:rho_P}. In Figure~\ref{fig:eigs_svds}, we demonstrate the numerical results for the spectral radius $\rho(\mathcal{P}_m)$ and $\| \mathcal{P}_m \|_2$ for the Apple, Brain and Venus benchmark models. The results show that the conditions in \eqref{eq:rho_P} hold for each benchmark model. This means that the conditions in \eqref{eq:rho_P} are reasonable assumptions.

\begin{figure}
\center
\begin{subfigure}[b]{0.49\textwidth}
\center
\includegraphics[width=\textwidth]{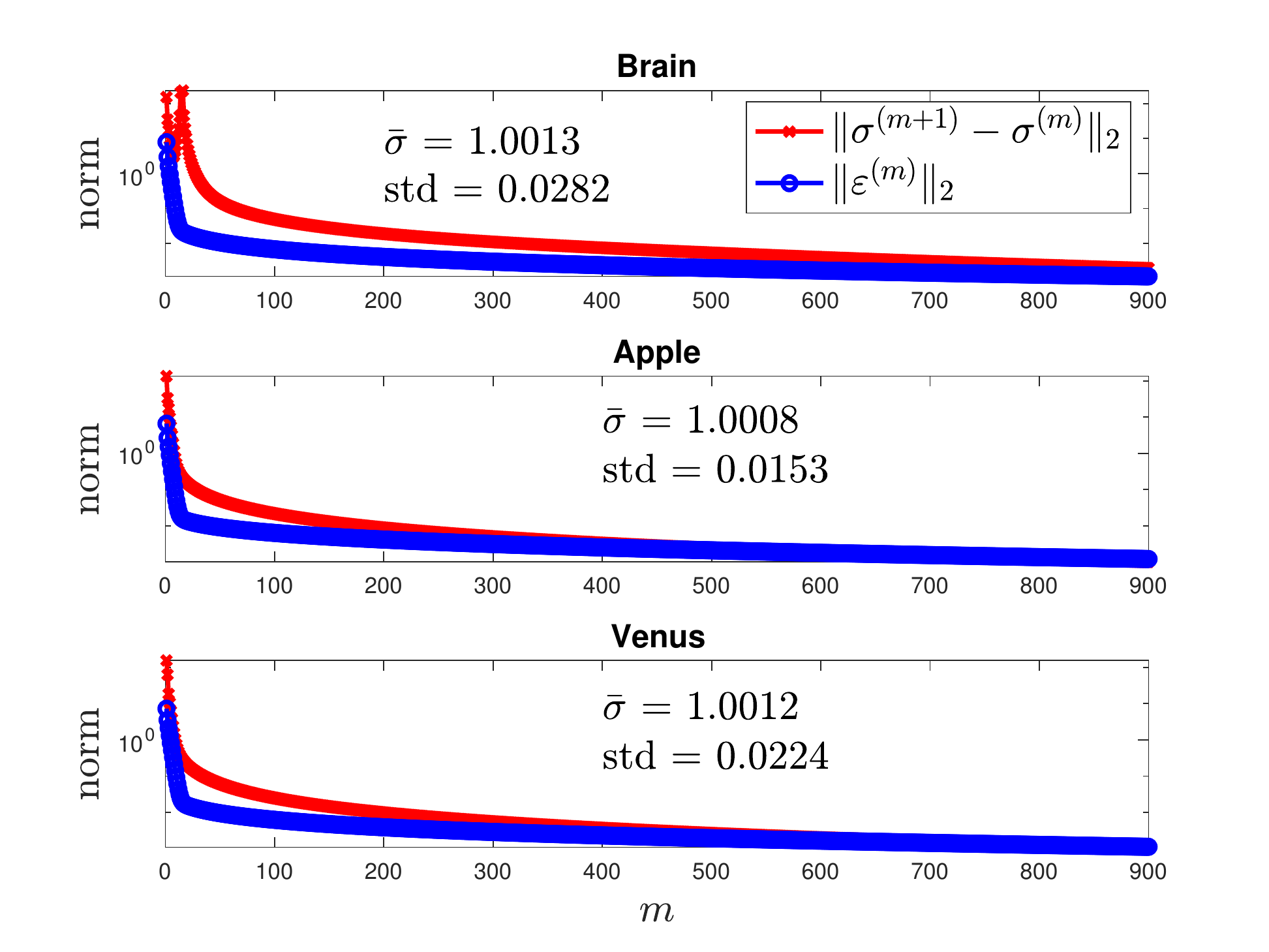}
\caption{$\| \beps^{(m)} \|_2$ and $\| \bsigma^{(m)} \|_2$}
\label{fig:conv_behav_brain_apple_venus}
\end{subfigure}
\begin{subfigure}[b]{0.49\textwidth}
\center
\includegraphics[width=\textwidth]{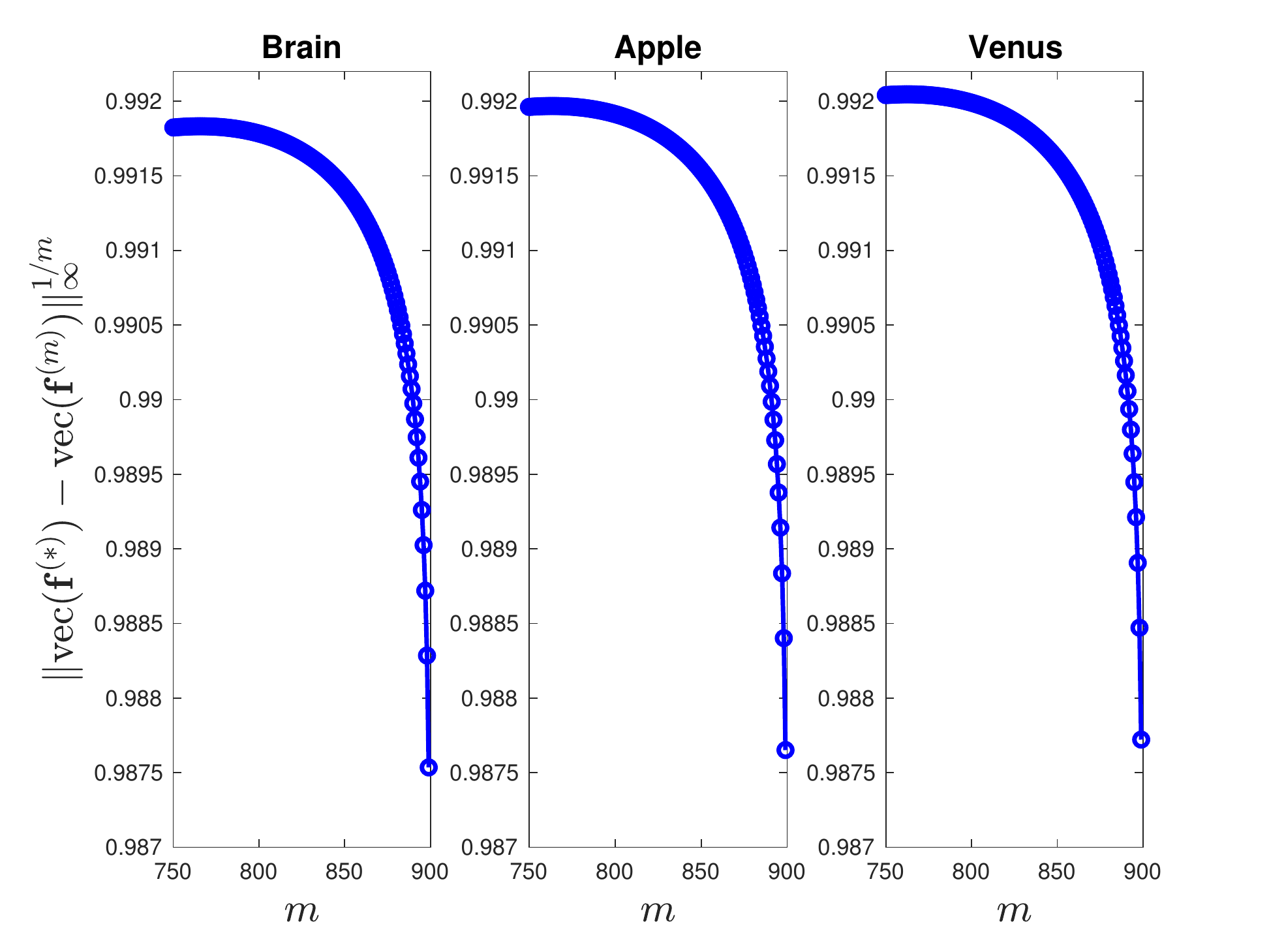}
\caption{$\|\mbox{vec}(\mathbf{f}^{(*)}) - \mbox{vec}(\mathbf{f}^{(m)}) \|_{\infty}^{1/m}$}
\label{fig:R_linear_brain_apple_venus}
\end{subfigure}
\caption{Convergence behavior of $\| \beps^{(m)} \|_2$ and stretch factor $\| \bsigma^{(m)} \|_2$ in (a) and the enlarged R-linear convergence results of $\mathbf{f}^{(m)}$ in (b) for the Apple, Brain and Venus benchmark problems. Here, $\bar{\sigma}$ and std represent the mean value and standard deviation of $\bsigma^{(900)}$, respectively.}
\label{fig:h_I2B2_R_linear_conv}
\end{figure}

Let $\bsigma^{(m)}$ denote the stretch factor vector
in \eqref{eq:stretch_factor} with $\mu(\tau) = | \tau |$ at the $m$th iteration and $\beps^{(m)}$ in \eqref{eq:reorder_eps} denote the error terms of $\f_{\I}^{s(m)}$ and $\f_{\I}^{s(m-1)}$.
In Figure~\ref{fig:conv_behav_brain_apple_venus}, we show the convergence behavior of $\| \beps^{(m)} \|_2$ and $\| \bsigma^{(m)} - \bsigma^{(m-1)} \|_2$. The numerical results tell us that both $\f_{\I}^{s(m)}$ and $\bsigma^{(m)}$ are convergent.
The mean value ($\bar{\sigma}$) and standard deviation (std) of $\bsigma^{(900)}$ for each benchmark model are shown in Figure~\ref{fig:conv_behav_brain_apple_venus}. Each $\bar{\sigma}$ is close to one, which means that the associated map is volume-preserving.

Take $\f_{\I}^{s(900)}$ as the convergence $\f_{\I}^{s(\ast)}$ and plot $\|\mbox{vec}(\mathbf{f}^{(*)}) - \mbox{vec}(\mathbf{f}^{(m)}) \|_{\infty}^{1/m}$ for various $m$. We show the results for $m = 700, \ldots, 900$ in Figure~\ref{fig:R_linear_brain_apple_venus}. These results demonstrate the R-linear convergence of the VSEM algorithm.

\begin{figure}
\center
\begin{subfigure}[b]{0.32\textwidth}
\center
\includegraphics[width=\textwidth]{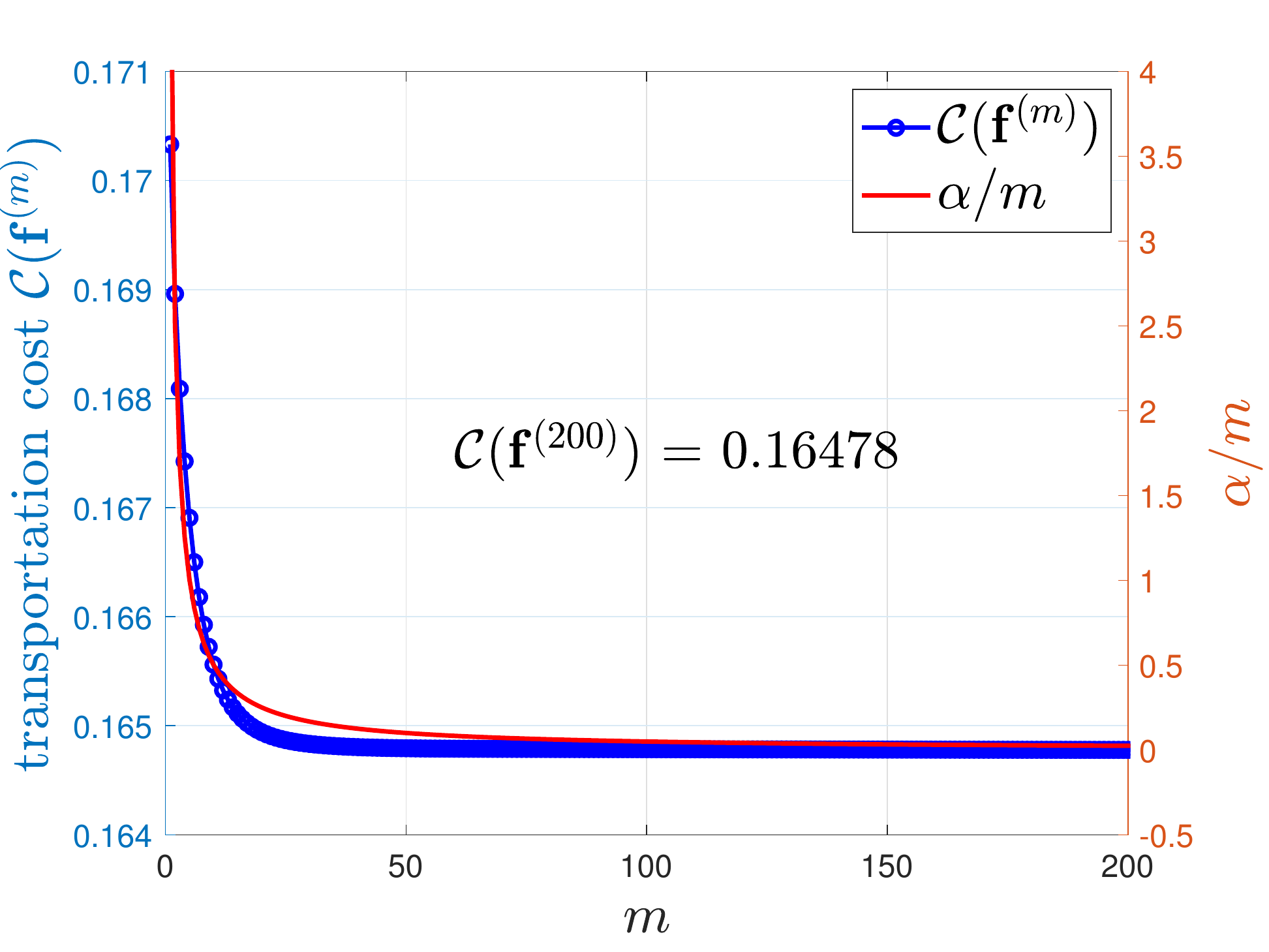}
\caption{David}
\label{fig:David_1_k_convergence}
\end{subfigure}
\begin{subfigure}[b]{0.32\textwidth}
\center
\includegraphics[width=\textwidth]{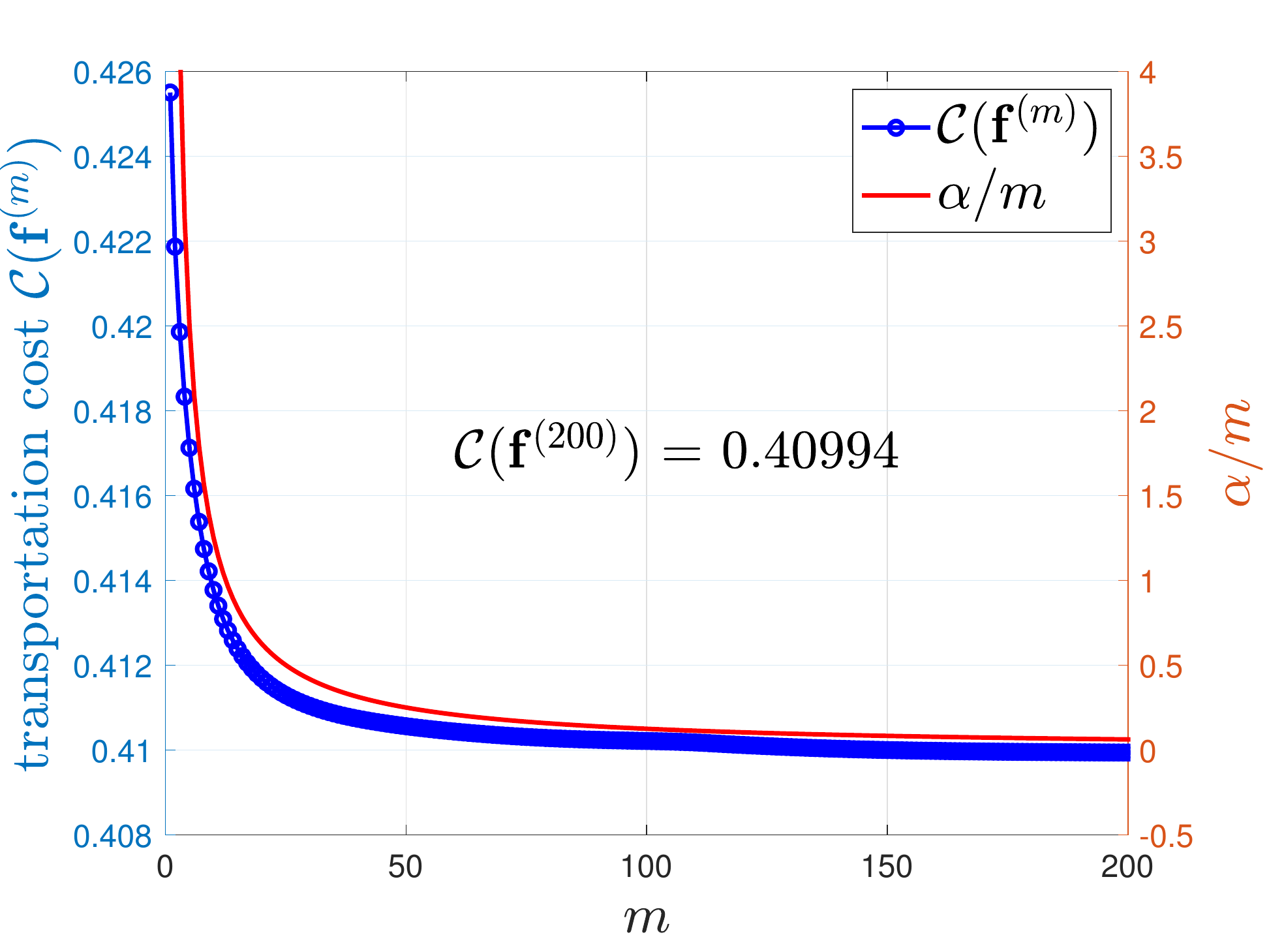}
\caption{Lion}
\label{fig:Lion_1_k_convergence}
\end{subfigure}
\begin{subfigure}[b]{0.32\textwidth}
\center
\includegraphics[width=\textwidth]{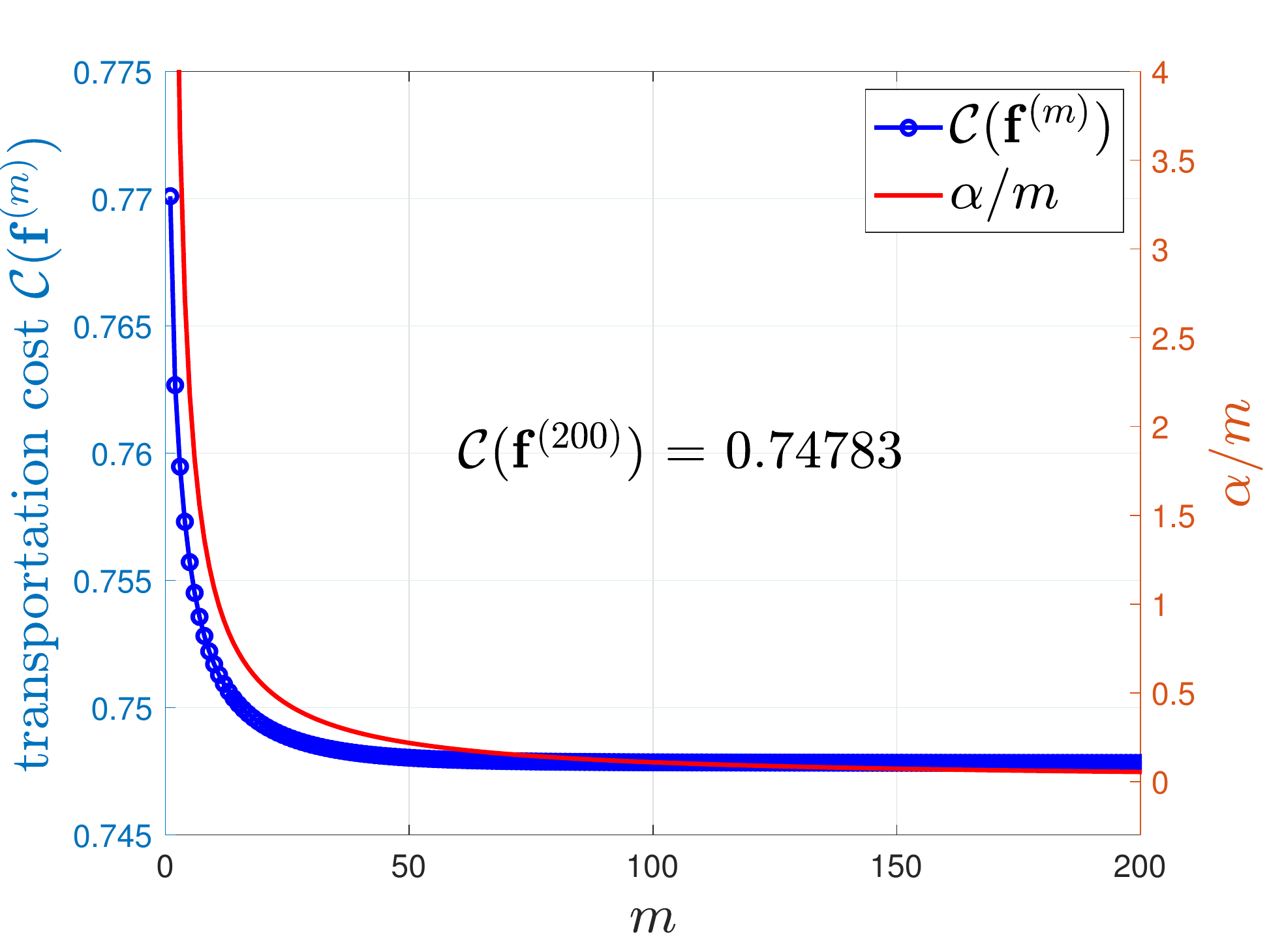}
\caption{Arnold}
\label{fig:Arnold_1_k_convergence}
\end{subfigure}
\caption{
The relationship between the number of iterations and the transportation cost of the projected gradient algorithm for computing the VOMT maps for David, Lion and Arnold benchmark models.}
\label{fig:1_k_convergence}
\end{figure}

\subsection{Convergence of the projected gradient method}
\label{subsec:6.3}

To verify the convergence of the projected gradient method in Algorithm~\ref{alg:VOMT} for computing the VOMT maps, in Figure \ref{fig:1_k_convergence} we demonstrate the relationship between the number of iterations and the transportation cost $\mathcal{C}(f)$ in \eqref{eq:cost} for the David, Lion and Arnold benchmark models. The convergence is, indeed, $\mathcal{O}(1/m)$, which is consistent with the conclusion of Theorem \ref{thm:ProjGrad}. The associated transportation costs at $m = 200$ of the David, Lion and Arnold benchmark models are $0.16478$, $0.40994$, and $0.74783$, respectively.

\begin{figure}
\center
\begin{subfigure}[b]{0.32\textwidth}
\center
\includegraphics[width=\textwidth]{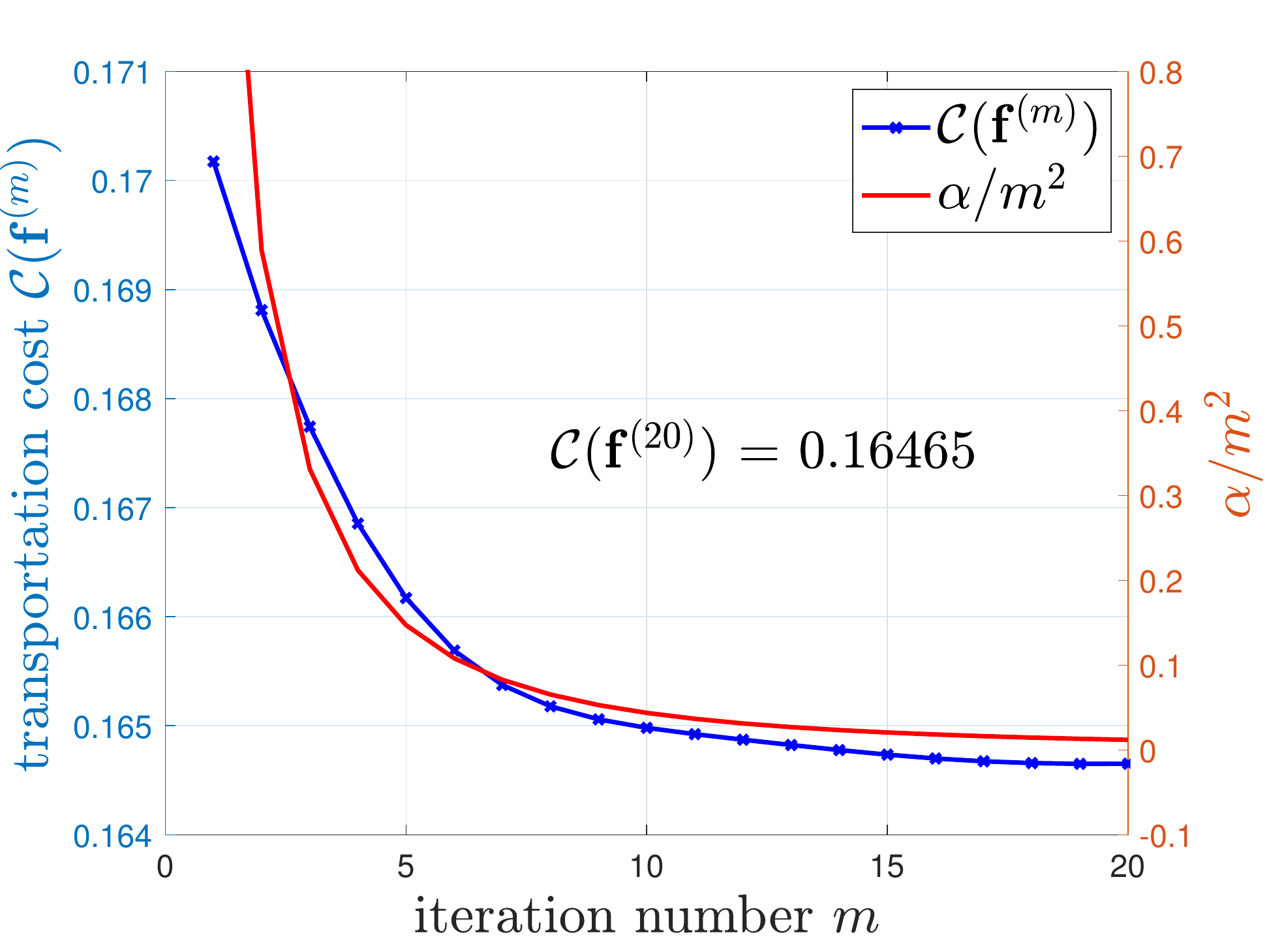}
\caption{David}
\label{fig:David_1_k2_convergence}
\end{subfigure}
\begin{subfigure}[b]{0.32\textwidth}
\center
\includegraphics[width=\textwidth]{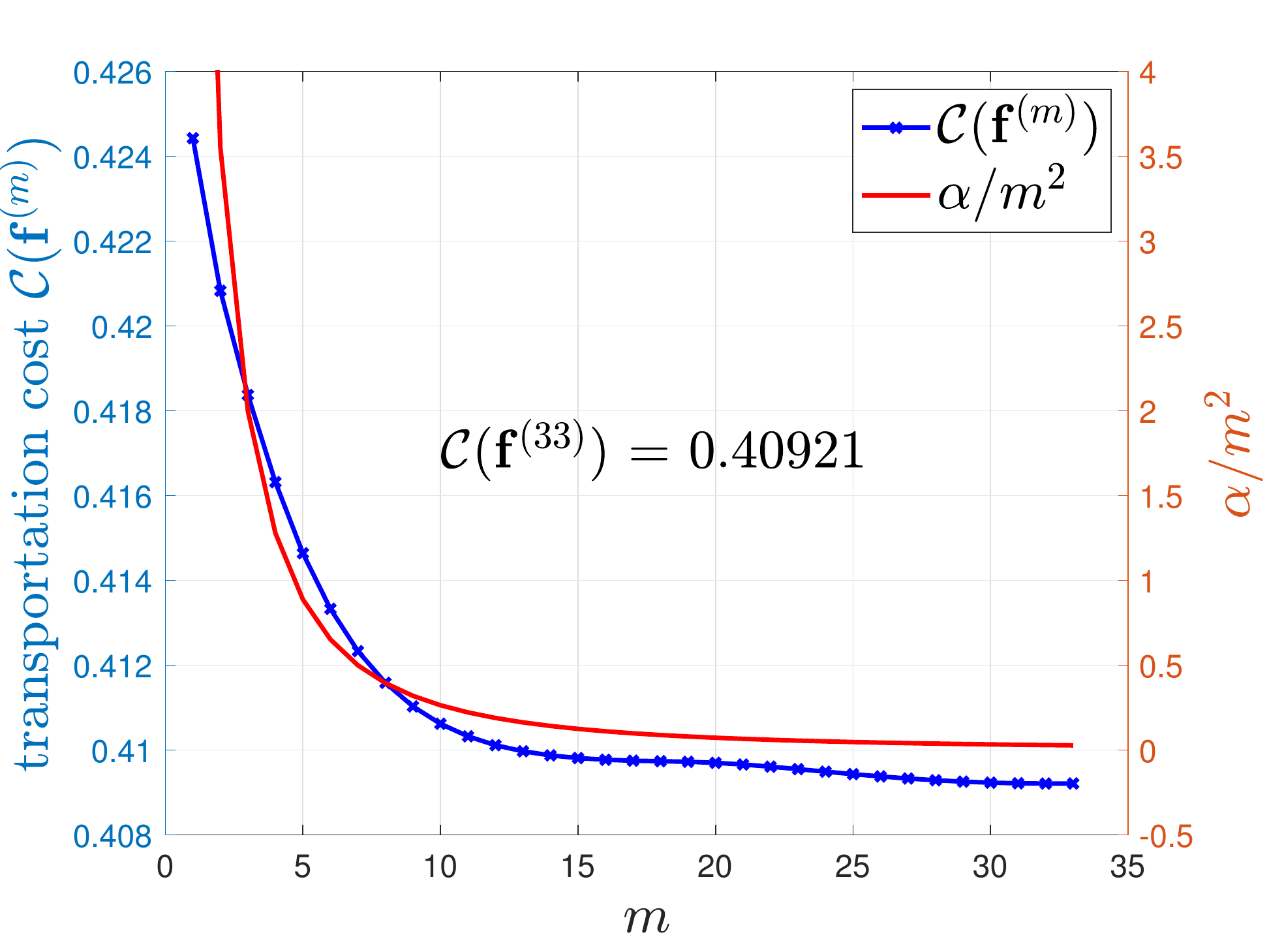}
\caption{Lion}
\label{fig:Lion_1_k2_convergence}
\end{subfigure}
\begin{subfigure}[b]{0.32\textwidth}
\center
\includegraphics[width=\textwidth]{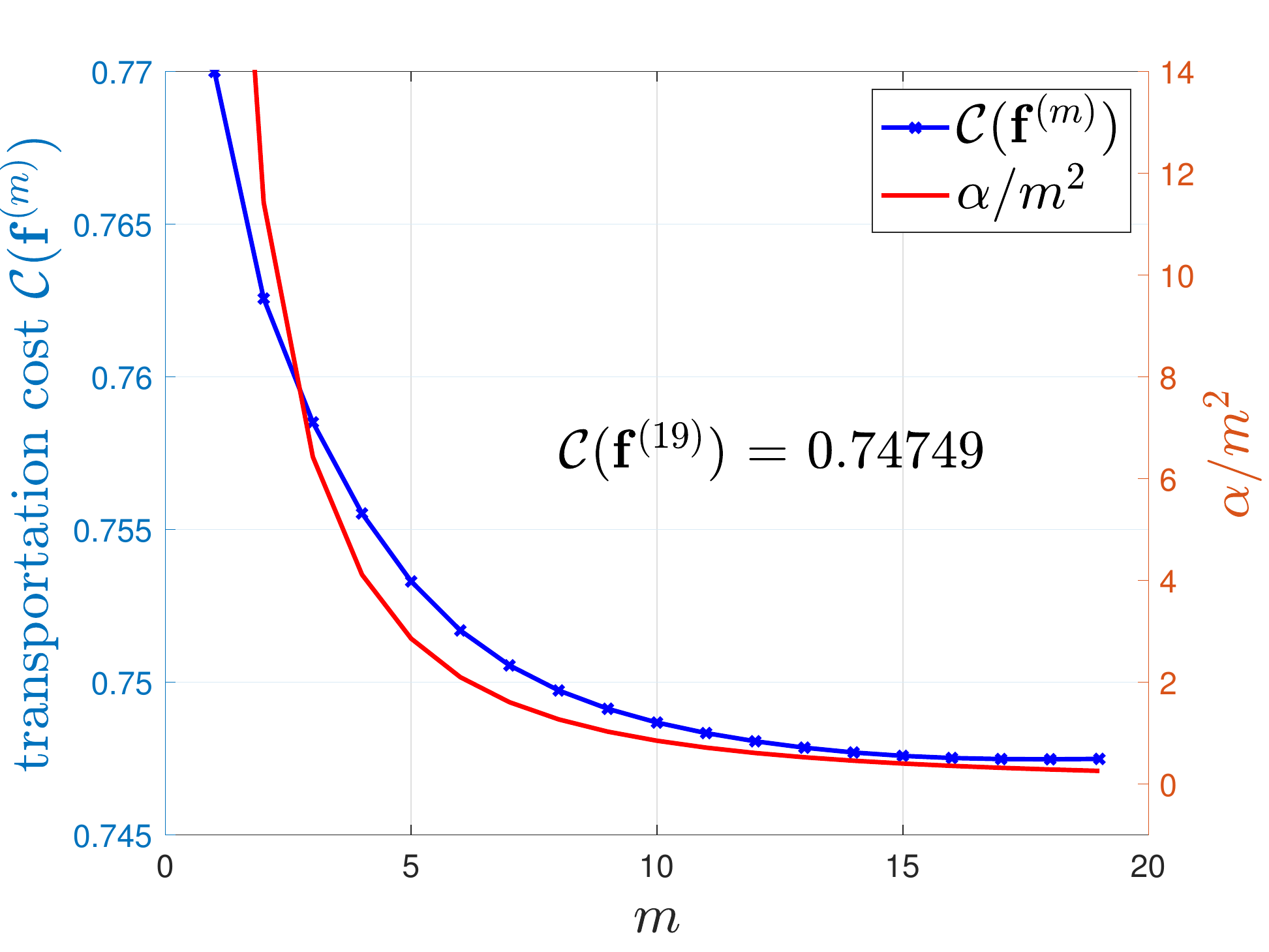}
\caption{Arnold}
\label{fig:Arnold_1_k2_convergence}
\end{subfigure}
\caption{
The relationship between the number of iterations and the transportation cost of the projected gradient algorithm with FISTA acceleration for computing the VOMT maps for David, Lion and Arnold benchmark models.}
\label{fig:1_k2_convergence}
\end{figure}

In Figure~\ref{fig:1_k2_convergence}, we show the transportation cost for computing the VOMT maps of the David, Lion, and Arnold benchmark models by using the projection gradient algorithm with FISTA acceleration. The results in the figure show that the algorithm converges in less than $35$ iterations, and the rate of convergence is $\mathcal{O}(1/m^2)$.
The associated transportation costs of the David, Lion and Arnold benchmark models are $0.16465$, $0.40921$, and $0.74749$, respectively, which reach smaller cost values than the $200$th iteration of the original projected gradient method in Figure \ref{fig:1_k_convergence}.
These results indicate that the FISTA indeed accelerates the convergence of the projection gradient algorithm.

\subsection{Practical implementations} \label{subsec:prac_impl}
In Subsections \ref{subsec:6.1} and \ref{subsec:6.3}, we give numerical validation for the theoretical convergences in Theorems~\ref{thm3.6} and \ref{thm:ProjGrad}. To reduce the computational cost, we will give some numerical observations to build up an efficient method for computing the VOMT map.

\begin{figure}
\centering
\begin{tabular}{cccc}
\includegraphics[height=4.5cm]{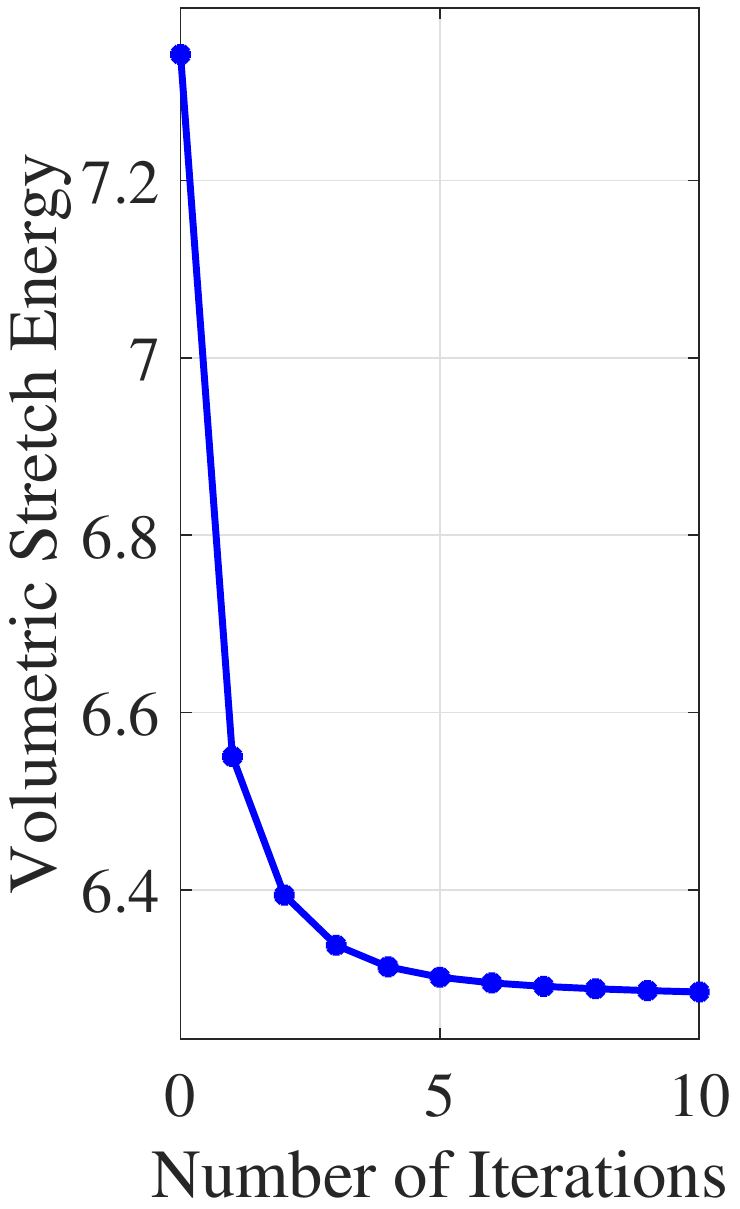} &
\includegraphics[height=4.5cm]{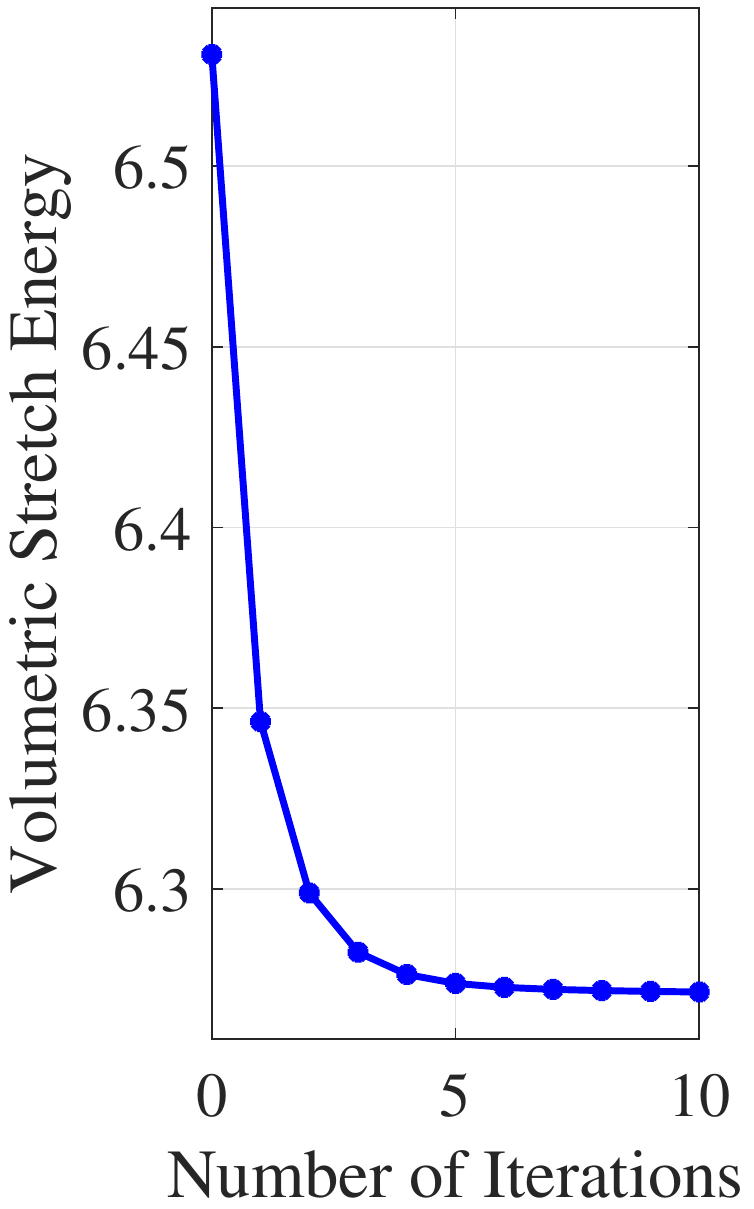} &
\includegraphics[height=4.5cm]{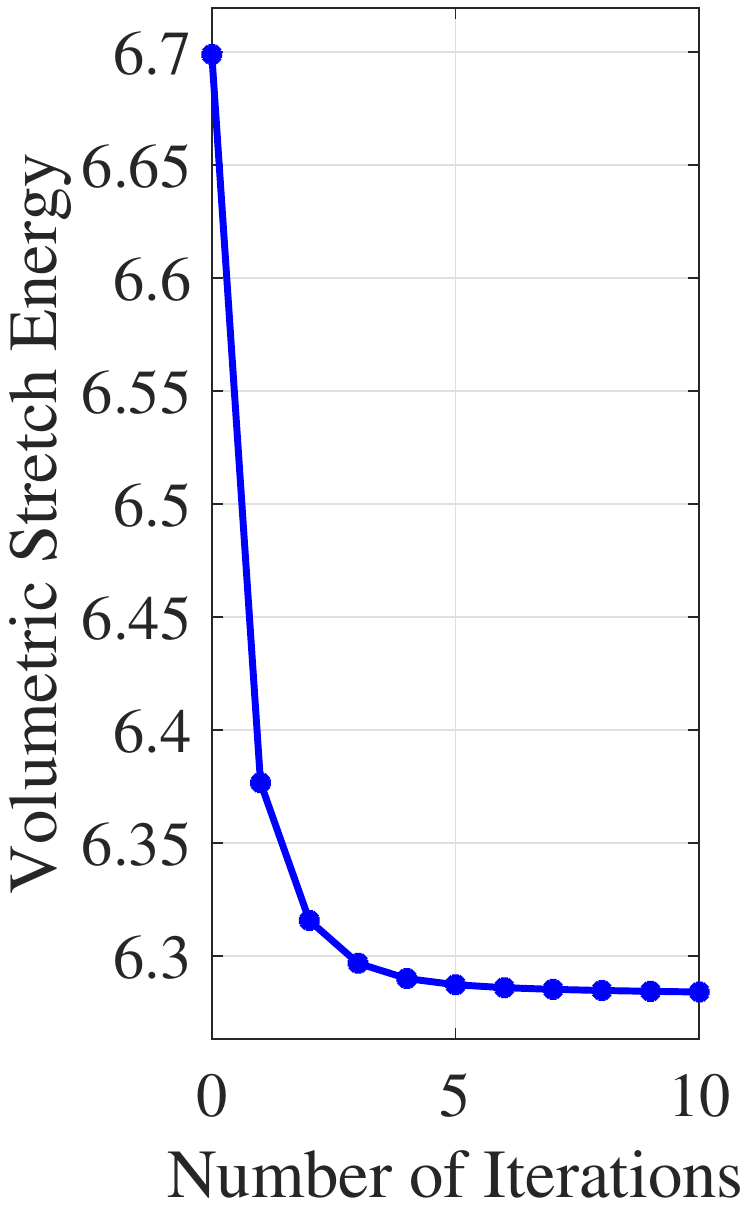} &
\includegraphics[height=4.5cm]{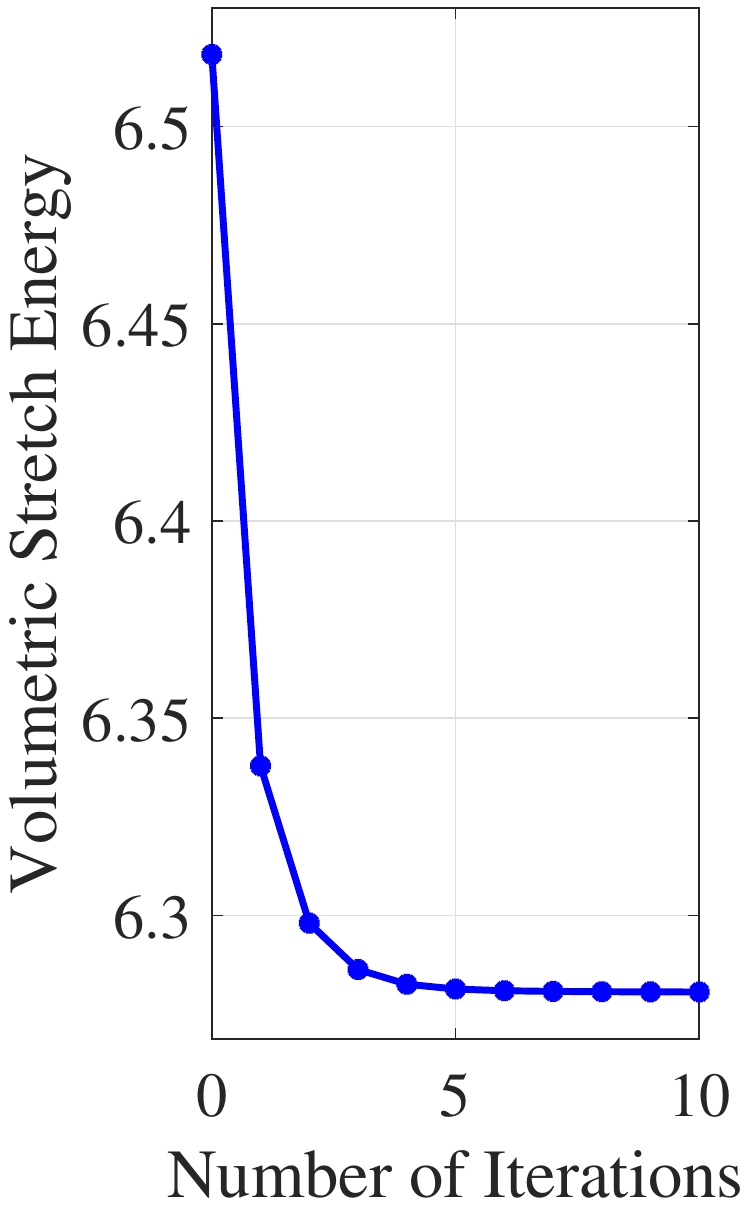} \\
(a) Arnold & (b) Igea & (c) David & (d) Apple
\end{tabular}
\caption{The relationship between the number of iterations and the volumetric stretch energy \eqref{eq:energy_fun} of the VSEM algorithm for computing volume-preserving parameterizations of the benchmark mesh models.}
\label{fig:E_V}
\end{figure}

To show the effectiveness of the VSEM algorithm in decreasing the energy \eqref{eq:energy_fun}, in Figure \ref{fig:E_V} we demonstrate the relationship between the number of iterations and the volumetric stretch energy for the benchmark mesh models. From the results in the figure, we observe that the energy decreases drastically in the first five steps and tends to converge to a constant after that. Therefore, we take the maximal iterations of Algorithm~\ref{alg:VSEM} to be $5$.

\begin{table}[]
\centering
\caption{The mean and std of the stretch factors, the volumetric stretch energy, computational time cost, and numbers of folding tetrahedra of the volume-preserving parameterization computed by the VSEM algorithm with 5 iterations.}
\label{tab:VSEM}
\begin{tabular}{lrrcccr}
\hline
\multirow{2}{*}{model}& \multirow{2}{*}{$\#(\mathbb{T}(\mathcal{M}))$} & \multirow{2}{*}{$\#(\mathbb{V}(\mathcal{M}))$} &  \multicolumn{2}{c}{stretch factor}  & \multirow{2}{*}{$E_V(f)$} & \# fold- \\
  &  &  &  mean  & std &  &  ings  \\ \hline
Arnold     &  36,875 &   6,990 & 1.0129 & 0.1716 & 6.3014 & 3 \\ 
Heart      & 103,751 &  18,408 & 1.0023 & 0.0551 & 6.2722 & 1 \\ 
Igea       & 130,375 &  22,930 & 1.0020 & 0.0462 & 6.2738 & 0 \\ 
David Head & 233,663 &  40,669 & 1.0034 & 0.0780 & 6.2873 &  0 \\ 
Max Planck & 390,361 &  66,935 & 1.0025 & 0.0814 & 6.2872 &  8 \\ 
Apple      & 559,122 & 102,906 & 1.0002 & 0.0201 & 6.2813 &  0 \\ \hline
\end{tabular}
\end{table}

In Table~\ref{tab:VSEM}, we demonstrate the mean and std of the stretch factors as well as the volumetric stretch energy of the volume-preserving parameterization computed by Algorithm~\ref{alg:VSEM} at the $5$th iteration. These results indicate that Algorithm~\ref{alg:VSEM} with $5$ iterations is effective in computing the desired volume-preserving mappings.

In addition, from Figure~\ref{fig:1_k2_convergence}, the transportation cost also drastically decreases in the first few steps. This indicates that we can set maximal iterations of the projected gradient method with FISTA acceleration to be $2$ to reduce the computational cost. Therefore, in practical implementation, the maximal iterations of the projected gradient method with FISTA acceleration (PGFISTA) are equal to $2$, and in each iteration of PGFISTA, Algorithm~\ref{alg:VSEM} with $5$ iterations is used to compute the desired volume-preserving mapping.

In Table~\ref{tab:FISTA}, we show the numerical results of the VOMT maps computed by the PGFISTA for the tetrahedral mesh models of human brains from the BraTS database \cite{BaAS17,BaGB21}.
We observe that the volume-preserving properties of the resulting VOMT maps are satisfactory, with the std of the stretch factors less than $0.1$ and the volumetric stretch energies close to $2\pi$. Additionally, it is worth noting that the resulting VOMT maps are very close to bijective with less than $0.001\%$ folding tetrahedra. These results make us more confident in the practical applications of VOMT maps in analyzing brain images.

\begin{table}
\caption{The mean and std of the stretch factors, the volumetric stretch energy, transportation cost value, and numbers of folding tetrahedra of the VOMT maps computed by the projected gradient algorithm with FISTA acceleration.}
\label{tab:FISTA}
\centering
\begin{tabular}{crrcccrr}
\hline
BraTS 20& \multirow{2}{*}{$\#(\mathbb{T}(\mathcal{M}))$} & \multirow{2}{*}{$\#(\mathbb{V}(\mathcal{M}))$} &  \multicolumn{2}{c}{stretch factor}  & \multirow{2}{*}{$E_V(f)$} & \multirow{2}{*}{$\mathcal{C}(f)$} & \# fold- \\
image no. &    &         & mean   & std    &        &        & ings \\ \hline
212 & 728,672 & 123,656 & 1.0009 & 0.0945 & 6.3325 & 0.0915 & 4 \\ 
234 & 767,428 & 129,842 & 1.0013 & 0.0979 & 6.3371 & 0.1140 & 2 \\ 
254 & 758,586 & 128,511 & 1.0010 & 0.0982 & 6.3360 & 0.0943 & 7 \\ 
258 & 701,419 & 119,090 & 1.0010 & 0.0914 & 6.3287 & 0.1054 & 6 \\ 
270 & 762,801 & 129,035 & 1.0021 & 0.0602 & 6.2905 & 0.0921 & 5 \\ 
277 & 753,014 & 127,580 & 1.0015 & 0.0933 & 6.3189 & 0.1111 & 6 \\ 
289 & 717,220 & 121,738 & 1.0025 & 0.0756 & 6.2942 & 0.0941 & 8 \\ 
294 & 734,164 & 124,593 & 1.0011 & 0.0992 & 6.3291 & 0.1129 & 7 \\ 
297 & 739,479 & 125,546 & 1.0011 & 0.0951 & 6.3330 & 0.0940 & 9 \\ 
306 & 720,254 & 122,412 & 1.0014 & 0.0918 & 6.3249 & 0.1104 & 7 \\ 
320 & 733,924 & 124,475 & 1.0013 & 0.0996 & 6.3224 & 0.0875 & 7 \\ 
321 & 655,973 & 111,732 & 1.0014 & 0.0965 & 6.3329 & 0.1266 & 7 \\ 
339 & 695,370 & 118,004 & 1.0012 & 0.0843 & 6.3192 & 0.1030 & 6 \\ 
344 & 736,101 & 124,691 & 1.0007 & 0.0847 & 6.3244 & 0.1076 & 2 \\ 
351 & 700,369 & 118,864 & 1.0015 & 0.0996 & 6.3211 & 0.1334 & 6 \\ 
352 & 736,959 & 124,868 & 1.0014 & 0.0977 & 6.3383 & 0.1138 & 2 \\ 
361 & 721,220 & 122,203 & 1.0012 & 0.0865 & 6.3163 & 0.1162 & 4 \\ 
364 & 740,970 & 125,510 & 1.0013 & 0.0854 & 6.3107 & 0.0959 & 3 \\ \hline
\end{tabular}
\end{table}

\begin{remark}
By applying the large-scale bounded distortion mapping \cite{KoAB15,LargeScaleBD}, the folding tetrahedra can be unfolded, slightly sacrificing the volume distortion, so that the resulting mapping is bijective. 
\end{remark}

\section{Concluding remarks}
\label{sec:7}

In this paper, we provided the theoretical foundation for discrete volumetric stretch energy minimization for computing volume-preserving parameterizations and developed the associated efficient VSEM algorithm with guaranteed R-linear convergence.
In addition, based on the VSEM algorithm, we proposed a projected gradient method with Nesterov-based acceleration for the computation of VOMT maps with a guaranteed $\mathcal{O}(1/m)$ convergence rate.
The associated numerical experiments were demonstrated to justify the consistency of the theoretical and numerical results.
Numerical results also showed the effectiveness and accuracy of the proposed VSEM and VOMT algorithms. Such encouraging results provide a solid foundation for further applications of volume-preserving parameterizations and OMT maps.

\section*{Acknowledgments}
We are grateful to Professor Tiexiang Li from Southeast University and Nanjing Center for Applied Mathematics for valuable discussions and for providing the medical 3D MRI brain images in Table \ref{tab:FISTA}.

\bibliographystyle{abbrv}
\bibliography{VSEM}

\end{sloppypar}
\end{document}